\documentclass[11pt]{article}    
\usepackage[margin=1in]{geometry}
\usepackage{graphicx}
\usepackage{amsmath,amsfonts,amsthm,amssymb}
\usepackage{mathrsfs}
\usepackage{hyperref}
\usepackage{caption}
\usepackage{subcaption}
\usepackage{comment}
\usepackage[shortlabels]{enumitem}

\numberwithin{equation}{section}
\newtheorem{thm}{Theorem}
\newtheorem{lem}{Lemma}[section]
\newtheorem{cor}[lem]{Corollary}
\newtheorem{prop}{Proposition}
\theoremstyle{definition}
\newtheorem{defn}{Definition}
\theoremstyle{remark}
\newtheorem{rem}{Remark}[section]

\newcommand{\pare}[1]{\left(#1\right)}
\newcommand{\brak}[1]{\left[#1\right]}

\newcommand{\Pare}[1]{\big(#1\big)}

\newcommand{\Brak}[1]{\big[#1\big]}

\newcommand{\abs}[1]{\left\lvert #1 \right\rvert}
\newcommand{\Abs}[1]{\big\lvert #1 \big\rvert}

\newcommand{\dd}[2]{\frac{\partial #1}{\partial #2}}

\newcommand{\AnD}{\quad\text{and}\quad}
\newcommand{\AND}{\qquad\text{and}\qquad}

\newcommand{\RR}{\mathbb{R}}
\newcommand{\Ss}{\mathbb{S}}

\newcommand{\NN}{\mathbb{N}}
\newcommand{\ZZ}{\mathbb{Z}}

\newcommand{\im}{\mathrm{i}}
\DeclareMathOperator{\sgn}{sgn}
\DeclareMathOperator{\Tr}{Tr}

\newcommand{\He}{\mathbf{H}}
\newcommand{\vT}{\theta}
\newcommand{\cS}{\mathcal{C}}
\newcommand{\rn}{\rho_n}

\newcommand{\Int}{\mathcal{I}}
\newcommand{\Cp}{\mathcal{T}}
\newcommand{\CP}{\widetilde{\Cp}}
\newcommand{\N}{\mathcal{N}}
\newcommand{\Gp}{\mathcal{G}}

\usepackage{xcolor}

\begin{document}
\title{Finiteness Results for Non-Scattering Herglotz Waves
	\\ \large The case of inhomogeneities obtained by very general perturbations of disks.}
\author{Michael S. Vogelius \footnote{Department of Mathematics, Rutgers University, New Brunswick and Department of Mathematics, Aarhus University} and Jingni Xiao \footnote{Department of Mathematics, Drexel University, Philadelphia}}
\date{}
\maketitle

\begin{abstract}
	We study non-scattering phenomena associated with the time-harmonic Helmholtz equation in two dimensions. For very general classes of star-shaped domains, we show that there are at most finitely many wave numbers such that Herglotz incident waves with a fixed density function are non-scattering.
\end{abstract}
\section{Introduction}
In this paper we study the {\it geometric} implications of (a high degree of) non-scattering in the context of the two dimensional Helmholtz equation. The focus is on smooth inhomogeneities, since it is already known that non-scattering, at just a single wave number and for a single incident wave, generically implies that the inhomogeneity (scatterer) is smooth \cite{CakVog23,SaShah,CaVoX23}. The present study may be seen as a continuation of the investigation initiated in \cite{VogX20}.
We consider the two dimensional Helmholtz equation for a wave number, $k>0$, and
an incident wave in the form of a superposition of plane waves, a so-called Herglotz wave:
\begin{equation}
	\label{herglotz}
	\He [k, \phi](x)=\int_{-\pi}^{\pi}\phi(\xi) e^{\im k\xi\cdot x}~d\theta_{\xi}~.
\end{equation}
The function $\phi\in L^2(\Ss^1)$ is the associated  Herglotz density. We note that $\He [k, \phi](x)$ solves
$\Delta \He [k, \phi] + k^2 \He [k, \phi]=0$ in all of  $\RR^2$. The bounded inhomogeneity $\Omega \subset \RR^2$, with index of refraction $q\in L^{\infty}(\Omega)$ (not identically $1$), is non-scattering  in the presence of this incident wave (at wave number $k$), if and only if the total field outside $\Omega$ coincides with the incident field. This happens if and only if there exists a solution to the following over-determined problem
\begin{equation}\label{eq:OverDet}
	\begin{cases}
		\Delta u + k^2 q u =0
		&\qquad
		\mbox{in $\Omega$},
		\\
		u-\He [k, \phi]=0,\quad \partial_{\nu}u-\partial_{\nu}\He [k, \phi]=0,
		&\qquad
		\mbox{on $\partial\Omega$}.
	\end{cases}
\end{equation}
The function $u$ represents the transmitted wave inside $\Omega$. In terms of the transmission eigenvalue problem
\begin{equation}\label{eq:ITEP}
	\begin{cases}
		\Delta u + \lambda q u =0,\quad
		\Delta v + \lambda v =0,
		&\qquad
		\mbox{in $\Omega$},
		\\
		u-v=0,\quad \partial_{\nu}u-\partial_{\nu}v=0,
		&\qquad
		\mbox{on $\partial\Omega$},
	\end{cases}
\end{equation}
this is equivalent to the fact that $\lambda = k^2$ is a transmission eigenvalue with eigenvector component $v$ given by $\He [k, \phi]$.
Hence the (squares of the) non-scattering wave numbers form a subset of the transmission eigenvalues. Regarding the latter it is known that, under minimal sign assumptions on $q-1$, there exists a countably infinite set of positive real eigenvalues for the problem \eqref{eq:ITEP}; See, \cite{CGH10,CCH16book}.

Throughout this paper we assume that {\bf$q$ is a  positive constant, different from $1$}. For disks centered at the origin we recall the following result concerning the existence of  infinite sequences of non-scattering (Herglotz) wave numbers, see \cite{ColKre19,CoM88,VogX20}.
\begin{prop}
	Let $\Omega$ be a disk of radius $R_0>0$ centered at the origin, and let $n\in\NN$ be any fixed integer. Then there exists an infinite sequence of positive wave numbers $\{k_j^{(n)}\}_{j=1}^{\infty}$ with $k_j^{(n)} \rightarrow \infty$ as $j \rightarrow \infty$ such that $\Omega$ is non-scattering for the incident Herglotz waves $\mathbf{H}[k_j^{(n)}, e^{-i n \theta}], j=1,2, \ldots$.
\end{prop}
For disks it is actually easy to see,  that the set of transmission eigenvalues coincides with the (squares of the) non-scattering wave numbers.

In this paper, we consider domains $\Omega$ that are not disks centered at the origin. We prove that, although there is a countably infinite set of positive real transmission eigenvalues for the problem \eqref{eq:ITEP}, there are at most finitely many non-scattering wave numbers corresponding to the Herglotz incident waves associated with a fixed density $\phi$.

Finally we notice that $\tilde{u}:=u-\He [k, \phi]$ satisfies
	\begin{equation}\label{eq:schiffer}
		\begin{cases}
			\Delta \tilde{u} + k^2 q \tilde{u} =f
			&\qquad
			\mbox{in $\Omega$},
			\\
			\tilde{u}=0,\quad \partial_{\nu}\tilde{u}=0,
			&\qquad
			\mbox{on $\partial\Omega$},
		\end{cases}
	\end{equation}
	with $f=-k^2(q-1)\He [k, \phi]$. The over-determined problem \eqref{eq:schiffer} resembles the so-called Schiffer problem. 
	The Schiffer conjecture states that, given $\Omega$ a simply connected domain in $\RR^m$, $m=2,3$, \eqref{eq:schiffer} admits a solution with $f\equiv 1$ for some $k\in \RR$ if and only if $\Omega$ is an $m$-dimensional ball. The proof of this is still an open problem, with some partial results, including that \eqref{eq:schiffer} admits a solution with $f\equiv 1$ for infinitely many $k\in \RR$ if and only if $\Omega$ is an $m$-dimensional ball \cite{BerYan87,Vog94}.
In this paper we show (for $m=2$) that for broad classes of non-circular domains, $\Omega$, and for a fixed nontrivial Herglotz density, $\phi$, there exist at most finitely many wave numbers $k$, for which $\Omega$ is non-scattering given the incident wave $\He [k, \phi]$. Precise statements of our results are given in the next section.

\subsection{Main Results}
One of our results concerns ellipses centered at the origin. We recall that the eccentricity of an ellipse is given by the formula $\sqrt{1-\frac{b^2}{a^2}}$, where $a$ is the length of its semi-major axis, and $b$ is the length of its semi-minor axis.
\begin{thm}\label{thm:ellipse}
	Let $\Omega$ be an \emph{ellipse} centered at the origin.
	Suppose that the eccentricity $e$ of $\Omega$ satisfies
	\begin{equation*}
		0<e^2<\sqrt{q}/(1+\sqrt{q}).
	\end{equation*}
	Then, given a $C
	^1$ function $\phi$ on $\Ss^1$, which is not identically zero,
	there are at most finitely many $k$'s such that $\Omega$ is non-scattering in the presence of the incident wave $v=\He[k,\phi]$. \footnote{We identify $\Ss^1$ with the interval $[0,2\pi)$ (and with the quotient space $\RR/(2\pi\ZZ)$) by the map $\theta\mapsto (\cos\theta,\sin\theta)^T$, unless otherwise specified.}
\end{thm}
This result, which was already announced in \cite{VogX20}, significantly extends the result of  Theorem 1 in that paper.
Another result concerns disks that are not centered at the origin.
\begin{thm}\label{thm:disk}
	Let $\Omega=B_{R_0}(x_0)$ be a \emph{disk} of radius $R_0$ centered at the point $x_0 \ne 0$.
	Suppose that
	\begin{equation*}
		|x_0|/R_0<\begin{cases}\sqrt{q}/(1+\sqrt{q})&\hbox{ for } \sqrt{q}<1 ~\text{or}~ 1< \sqrt{q}\le \frac1{\sqrt{3}-1}~,\\
		\sqrt{1-2(1+\sqrt{q})/(3\sqrt{3q} )}~~~\left(< \sqrt{q}/(1+\sqrt{q})\right) &\hbox{ for } \sqrt{q}> \frac1{\sqrt{3}-1}~.\end{cases}
	\end{equation*}
	Then, given a $C^1$ function $\phi$ on $\Ss^1$,  which is not identically zero,
	there are at most finitely many $k$'s such that $\Omega$ is non-scattering in the presence of the incident wave $v=\He[k,\phi]$.
\end{thm}
The following result concerns ellipses where the origin is one of the foci.
\begin{thm}\label{thm:ellipse2focus}
	Let $\Omega$ be an \emph{ellipse} with one of its foci located at the origin.
	Suppose the eccentricity $e$ of $\Omega$ satisfies
	\begin{equation*}
		\begin{cases}
			0<e^2<3/4-1/(4\sqrt{q}),&\qquad\mbox{if $q>1$ or $1/4<q<1$},
			\\0<e^2< q,&\qquad\mbox{if $0<q\le1/4$}.
		\end{cases}
	\end{equation*}
	Then, given a $C
	^1$ function $\phi$ on $\Ss^1$, which is not identically zero,
	there are at most finitely many $k$'s such that $\Omega$ is non-scattering in the presence of the incident wave $v=\He[k,\phi]$.
\end{thm}
Theorems~\ref{thm:disk}~and~\ref{thm:ellipse2focus} are direct consequences of Theorem~\ref{thm:main1Star} (and its counterpart for the case when $0<q<1$, see \cite{Xiao24}), which will be stated shortly. In fact, applying Theorem~\ref{thm:main1Star} we can also show similar results for certain off-center ellipses. For brevity we opt not to present those results. We postpone the proofs of Theorems~\ref{thm:disk}~and~\ref{thm:ellipse2focus} to the Appendix.

We now proceed to present our more general results pertaining to non-circular domains that are star-shaped with respect to the origin. For simplicity we only formulate and prove these  more general results for the case $q>1$, but similar results hold for $0<q<1$, see \cite{Xiao24}. Since Theorems~\ref{thm:disk}~and~\ref{thm:ellipse2focus} are derived as consequences of Theorem~\ref{thm:main1Star}, they are here in reality only verified for $q>1$.  We shall need two definitions.
\begin{defn}
	Given a bounded and simply connected $C^2$ domain $\Omega\subset\RR^2$ containing the origin, we say that $\Omega$ is \emph{star-shaped with respect to the origin} if $\partial \Omega$ admits a parameterization $\{\rho(\theta)\vec{\theta}:\theta\in[0,2\pi)\}$, where $\vec{\theta}=(\cos\theta,\sin\theta)^T$ and $\rho$ is a positive $2\pi$-periodic $C^2$ function. We refer to $\rho$ as the \emph{radius function} of the domain $\Omega$.
\end{defn}

\begin{defn}
	Given a constant $q>1$ we call a \emph{radius function}, $\rho$, \emph{admissible} (with respect to $q$) if it satisfies the following ``smallness'' condition:
		\begin{equation}\label{eq:StarBigCond1}
			(\ln\rho)''(t)<\sqrt{q}/(1+\sqrt{q})\qquad\mbox{for all $t$}.
		\end{equation}
\end{defn}
Note that we only impose an upper bound for $(\ln\rho)''$, besides the natural condition $\int_{0}^{2\pi}(\ln\rho)''(t)dt=0$ due to periodicity. In particular, the quantity $-(\ln\rho)''(t)$ is allowed to be arbitrarily large in an arbitrarily small interval.

Our first result for these star-shaped domains concerns domains where ``half" of the boundary consists of circular arcs.
\begin{thm}\label{thm:Star1}
	Suppose $q>1$ and suppose  $\Omega$ is a $C^2$ domain, which is \emph{star-shaped} with respect to the origin with an admissible radius function $\rho$. Assume that $\rho$ satisfies
	\begin{equation*}
		\mbox{$\rho'=0$\quad in $\N$, \AND $\rho'\neq 0$\quad a.e. in $\N+\pi$,}
	\end{equation*}
	for some relatively open subset $\N$ of $[0,2\pi)$ satisfying $|\N|=\pi$ and $\overline{\N\cup(\N+\pi)}=[0,2\pi)$ (in the sense of the quotient space $\RR/(2\pi\ZZ)$). Let $\phi$ be a nontrivial $C^1$ function on $\Ss^1$.
	Then there are at most finitely many wave-numbers $k$ such that $\Omega$ is non-scattering in the presence of the incident wave $v=\He[k,\phi]$.
\end{thm}

\begin{rem}
	An example of $\Omega$ in Theorem~\ref{thm:Star1} is a star-shaped domain with the radius function
	\begin{equation*}
		\rho(t)=
		\begin{cases}
			1+a\sin^3 t ,&\qquad 0<t<\pi,
			\\
			1,& \qquad \pi<t<2\pi,
		\end{cases}
	\qquad \mbox{with $a\in(0,1)$ being a constant.}
	\end{equation*}
	$\Omega$ forms a $C^2$ ``egg'' shape (see first frame of Figure 1). One can verify that $0<5a\le\sqrt{q}/(1+\sqrt{q})$ is a sufficient condition for \eqref{eq:StarBigCond1} to be satisfied. We show in Figure~\ref{fig:1} some domains $\Omega$ that satisfy the conditions in Theorem~\ref{thm:Star1}.
	\begin{figure}[h]
		\centering
		\includegraphics{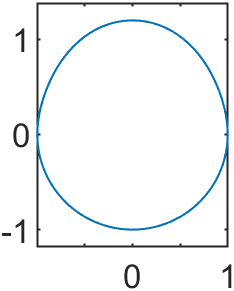}\hfil\includegraphics{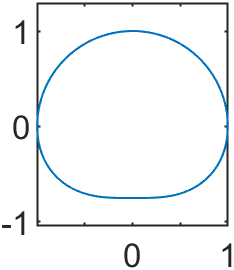}\hfil\includegraphics{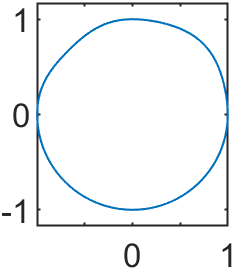}\hfil\includegraphics{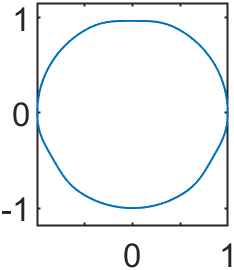}
		\caption{Examples of star-shaped domains fulfilling conditions in Theorem~\ref{thm:Star1}. From left to right: $\rho(t)=1+a\chi_{(0,\pi)}\sin^3t$, $\rho(t)=1+a\chi_{(\pi,2\pi)}\sin^3t$, $\rho(t)=1+a\chi_{(0,\pi)}\sin^3 2t$,  $\rho(t)=1+a\chi_{(0,\pi/3)\cup(2\pi/3,\pi)\cup(4\pi/3,5\pi/3)}\sin^3 3t$, where the values of $a\in(0,1)$ are chosen to be small so that \eqref{eq:StarBigCond1} is satisfied for all $q>1$.}
		\label{fig:1}
	\end{figure}
\end{rem}
The following result is the ``complement'' of Theorem~\ref{thm:Star1} in the sense that, it concerns domains with boundaries containing no circular arcs.
\begin{thm}\label{thm:main1Star}
Suppose $q>1$ and suppose  $\Omega$ is a $C^2$ domain, which is \emph{star-shaped} with respect to the origin with an admissible radius function $\rho$, that satisfies $\rho'\neq0$ a.e. in $[0,2\pi)$, and for any $t$ $\rho'(t)=0$ if and only if $\rho'(t+\pi)=0$. Assume furthermore that $\rho$ satisfies {\bf one} of the following two conditions:
	\begin{enumerate}[(i)]
		\item For each $t$ we have $\sgn \rho'(t)=\sgn\rho'(t+\pi)$, or
		\item For each $t$ such that $\rho'(t)=0$, we have $|\rho''(t)|+ |\rho''(t+\pi)|>0$.
	\end{enumerate}
Let $\phi$ be a nontrivial $C^1$ function on $\Ss^1$. Then there are at most finitely many wave numbers $k$ such that $\Omega$ is non-scattering in the presence of the incident wave $v=\He[k,\phi]$.
\end{thm}
\begin{rem}
	A particular case, where the condition (i) in Theorem~\ref{thm:main1Star} is satisfied,  is when $\rho$ is $\pi$-periodic, namely, when $\Omega$ is \emph{symmetric with respect to the origin}.
\end{rem}
\begin{rem}
		We note that the ellipses considered in Theorem~\ref{thm:ellipse} are star-shaped with respect to the origin, with radius functions
		\begin{equation*}
			\rho(t)=\frac{b}{\sqrt{1-e^2\cos^2t}},
		\end{equation*}
	and hence
	\begin{equation*}
		(\ln\rho)'(t)=\frac{-e^2\cos t\sin t}{1-e^2\cos^2t}
		\AND
(\ln\rho)''(t)=e^2\frac{1-(2-e^2)\cos^2t}{(1-e^2\cos^2t)^2}.
	\end{equation*}
We can verify that
\begin{equation*}
	\max_{t}\,(\ln\rho)''(t)=\begin{cases}(\ln\rho)''|_{\cos t=0}=
		e^2,&\qquad\mbox{if $0<e^2\le 2/3$},
		\\(\ln\rho)''|_{\cos^2 t=\frac{3e^2-2}{e^2(2-e^2)}}=\dfrac{(2-e^2)^2}{8(1-e^2)}>e^2,&\qquad\mbox{if $2/3<e^2<1$}.
	\end{cases}
\end{equation*}
In particular, concerning ellipses centered at the origin, Theorem~\ref{thm:ellipse} is more general than Theorem~\ref{thm:main1Star} when ($q>1$ and) the eccentricity satisfies $2/3<e^2<1$, in view of the admissibility condition \eqref{eq:StarBigCond1}.
\end{rem}
The following result concerns Herglotz waves with real-analytic densities.
\begin{thm}\label{thm:StarAnly}
Suppose $q>1$ and suppose  $\Omega$ is a $C^2$ domain, which is \emph{star-shaped} with respect to the origin with an admissible radius function $\rho$. Let $\phi$ be a
nontrivial function on $\Ss^1$, with $|\phi|^2$ \emph{real-analytic}. Furthermore assume that $\rho$ satisfies {\bf one} of the following two conditions:
	\begin{enumerate}[(i)]
		\item There exists some $t_0$ such that $\rho'(t_0)\neq0$ and $\rho'(t_0)\rho'(t_0+\pi)\ge 0$, or
		\item The quantity $|\rho'(t)|+|\rho'(t+\pi)|+ |\rho''(t)| +|\rho''(t+\pi)|$ is strictly positive for all $t$.
	\end{enumerate}
	Then there are at most finitely many wave numbers $k$ such that $\Omega$ is non-scattering in the presence of the incident wave $v=\He[k,\phi]$.
\end{thm}
\begin{rem}
	Note that we do not require $\rho'\ne 0$ a.e. in Theorem \ref{thm:StarAnly} (as in Theorem \ref{thm:main1Star}). Moreover, conditions (i) and (ii) in Theorem \ref{thm:StarAnly} are each much weaker than conditions (i) and (ii) in Theorem \ref{thm:main1Star}. However, the condition on $\phi$ is of course much more restrictive in Theorem~\ref{thm:StarAnly}.
\end{rem}

\subsection{Preliminaries}
If the overdetermined boundary value problem \ref{eq:OverDet} has a solution $u\in H^1(\Omega)$, then integration by parts yields, for any $w$ satisfying $\Delta w+k^2qw=0$ in $\Omega$, that
\begin{eqnarray*}
	k^2\int_{\Omega}(q-1)\,\He[k,\phi]\,w\,dx
	&=&\int_{\partial\Omega} \left( w\,\partial_{\nu} \He[k,\phi]- \He[k,\phi]\,\partial_{\nu} w~ \right) d\sigma(x) \\
	&=&\int_{\partial\Omega} \left( w\,\partial_{\nu} u- u \,\partial_{\nu} w~\right) d\sigma(x)=0~,
\end{eqnarray*}
Since $q$ is a positive constant different from $1$, we may substitute $w(x)=e^{\im k \sqrt{q}\, x\cdot \eta}$, for an arbitrary $\eta\in\Ss^1$, into this identity and obtain
\begin{equation*}
	k^2 \int_\Omega e^{\im k\sqrt{q}\,\eta\cdot x}\int_{-\pi}^{\pi}\phi(\xi)\, e^{\im k\xi\cdot x}\,d\vT_{\xi}\,dx
	=0,\quad\mbox{for any $\eta\in \Ss^1$},
\end{equation*}
or equivalently, via integration by parts,
\begin{equation}
	\label{integralB}
	\im {k} \int_{\partial\Omega}\int_{-\pi}^{\pi}
	\nu \cdot \pare{\xi-\sqrt{q}\eta} \phi(\xi)\,
	e^{\im k \pare{\sqrt{q}\eta+\xi}\cdot x} ~d\vT_{\xi}\,d\sigma(x)
	=0,\quad\mbox{for any $\eta\in \Ss^1$}.
\end{equation}
Here and in the following, for any vector $\xi\in\Ss^1$, we denote by $\vT_{\xi}$ the angular coordinate of $\xi$; in other words $\xi=(\cos\vT_{\xi},\sin\vT_{\xi})^T$.

Let $\Omega$ be a simple connected $C^{2}$ domain. By a parameterization $y=y(\theta)$ of $\partial\Omega$, we understand a bijective orientation-preserving $C^2$ mapping  $y:\Ss^1 \rightarrow \partial \Omega$  with $y'(\theta)\neq 0$ for any $\theta$. Moreover, the orientation of the parameterization is set to be counterclockwise.
We denote $x^{\perp}=(-x_2,x_1)^T$ for any vector $x=(x_1,x_2)^T\in\RR^2$.
With these conventions, the outwards unit normal vector to $\partial \Omega$ is given by $-y'^\perp/|y'|$.
We can now rewrite \eqref{integralB} as 
\begin{equation}\label{eq:IntFinal}
	\Int(k)=\Int(k;\eta;\Omega):= \int_{-\pi}^{\pi} \int_{-\pi}^{\pi}
	\Psi_\eta(\theta,\vT_{\xi})
	\,\phi(\xi)\,e^{\im k\,\psi_\eta(\theta,\vT_{\xi})} ~d\vT_{\xi}\,d\theta=0,\quad\mbox{for any $\eta\in \Ss^1$},
\end{equation}
where
\begin{equation}\label{eq:psi}
	\psi_\eta(\theta,\vT_{\xi})= \pare{\sqrt{q}\,\eta+\xi}\cdot y(\theta)
	\AND
	\Psi_\eta(\theta,\vT_{\xi})
	=-\pare{\sqrt{q} \,\eta-\xi }\cdot y'^{\perp}(\theta).
\end{equation}

We aim to show that, for a wide variety of domains $\Omega$ excluding disks centered at the origin, and for a fixed nontrivial $\phi$, there exist at most finitely many wave numbers $k$ such that $\Omega$ is non-scattering in the presence of the incident wave  $\He[k,\phi]$.
We achieve this with a proof by contradiction, starting with the assumption that there is a sequence of infinitely many such wave numbers. Then, as the square of these wave numbers are transmission eigenvalues (for the problem \eqref{eq:ITEP}) they must accumulate at $\infty$ (see, \cite{CCH16book}).
Consequently, we may consider the asymptotics, as $k\to\infty$, of the Fourier type integral $\Int(k)$ in \eqref{eq:IntFinal}, by using  the method of stationary phase. Finally we obtain a contradiction between the asymptotics and the identity \eqref{eq:IntFinal}.

For any fixed $\eta\in\Ss^1$, a stationary point $(\theta,\vT_{\xi})$ associated with the phase function $\psi_\eta$ is characterized by
\begin{equation}\label{eq:stationaryGeneral}
	\partial\psi_\eta/\partial\theta=\pare{\sqrt{q}\,\eta+\xi}\cdot y'(\theta)=0
	\AND
	\partial\psi_\eta/\partial\vT_{\xi}=	\xi^{\perp}\cdot y(\theta)	=	0,
\end{equation}
or equivalently,
\begin{equation}\label{eq:stationaryGeneral2}
	\xi=\pm y(\theta)/|y(\theta)|
	\AND
	\pare{\sqrt{q}\,\eta+\xi}\cdot y'(\theta)=0
	.
\end{equation}
Let $\cS=\cS_{\Omega,\eta}$ be the set of all stationary points $(\theta,\vT_{\xi})$ of the phase function $\psi_\eta$ for a fixed $\eta$.
We assume for the time being that $\cS$ is a nonempty finite set and that all elements in $\cS$ are simple stationary points. The latter is satisfied if and only if $\det D^2\psi_\eta(\theta,\vT_\xi)\neq 0$ for all $(\theta,\vT_{\xi})\in\cS$, where $D^2\psi_\eta$ denotes the Hessian
\begin{equation}\label{eq:Hessian}
	D^2\psi_\eta(\theta,\vT_{\xi})=\begin{bmatrix}
		\pare{\sqrt{q}\,\eta+\xi}\cdot y''(\theta)& \xi^{\perp}\cdot y'(\theta)
		\\
		\xi^{\perp}\cdot y'(\theta)&-\xi\cdot y(\theta)
	\end{bmatrix}.
\end{equation}
Thanks to the periodicity of all the involved functions, the leading order contribution $\Int(k;\eta;\Omega)$, as $k \rightarrow \infty$,  comes from the stationary points (see, \cite[Section 8.4]{BleHan86})
\begin{equation}\label{eq:asympGeneral}
	\Int(k;\eta;\Omega)=
	\frac{	2 \pi}{k} \sum_{
		(\theta,\vT_\xi)\in\cS
	} \text{\small$\frac{\phi(\vT_\xi)\,\Psi_\eta(\theta,\vT_\xi)}{\abs{\det D^2\psi_\eta(\theta,\vT_\xi)}^{1/2}}$}
	e^{\frac{i\pi}4\sgn D^2\psi_\eta(\theta,\vT_\xi)}e^{ik\psi_\eta(\theta,\vT_\xi)}
	+o\Pare{\frac{1}{k}},\quad\mbox{as $k\to\infty$},
\end{equation}
where $\sgn M$ is the signature of a matrix $M$, namely, the difference between the number of positive and negative eigenvalues of $M$.

In addition, we can differentiate \eqref{eq:IntFinal} with respect to $\vT_{\eta}$ and obtain for all $\eta\in \Ss^1$ that
\begin{equation*}
	\Int^{(1)}(k;\eta;\Omega) :=\int_{-\pi}^{\pi} \int_{-\pi}^{\pi}
	\Pare{\im k\sqrt{q}\eta^{\perp}\cdot y(\theta)\Psi_\eta(\theta,\vT_{\xi})
		-\sqrt{q}\eta\cdot y'(\theta)}\phi(\xi)\,e^{\im k\,\psi_\eta(\theta,\vT_{\xi})} ~d\vT_{\xi}\,d\theta=0.
\end{equation*}
Similar to \eqref{eq:asympGeneral}, we also have the following asymptotics for $\Int^{(1)}(k)$ as $k\to\infty$:
\begin{equation*}
	\Int^{(1)}(k;\eta;\Omega)=
	2 \sqrt{q} \pi \im \sum_{
		(\theta,\vT_\xi)\in\cS
	} \text{\small$\frac{\phi(\vT_\xi)\,\Psi_\eta(\theta,\vT_\xi)}{\abs{\det D^2\psi_\eta(\theta,\vT_\xi)}^{1/2}}$}\eta^{\perp}\cdot y(\theta)
	e^{\frac{i\pi}4\sgn D^2\psi_\eta(\theta,\vT_\xi)}e^{ik\psi_\eta(\theta,\vT_\xi)}
	+o\pare{1}.
\end{equation*}
Following this idea, we can differentiate \eqref{eq:IntFinal} $N$ times and denote the resulting quantity $\Int^{(N)}(k)$. By doing so we obtain, for each $N\in\mathbb{N}$ and $\eta\in\Ss^1$, that
\begin{equation}\label{eq:asympGen_diffeta}
	\begin{split}
		\Int^{(N)}(k;\eta;\Omega)=
		&\frac{	2 \pi}{k}\pare{\im k\sqrt{q}}^N  \sum_{
			(\theta,\vT_\xi)\in\cS
		} \text{\small$\frac{\phi(\vT_\xi)\,\Psi_\eta(\theta,\vT_\xi)}{\abs{\det D^2\psi_\eta(\theta,\vT_\xi)}^{1/2}}$}\Pare{\eta^{\perp}\cdot y(\theta)}^N
		e^{\frac{i\pi}4\sgn D^2\psi_\eta(\theta,\vT_\xi)}e^{ik\psi_\eta(\theta,\vT_\xi)}
		\\&\quad+o\pare{k^{N-1}},
		\qquad\mbox{as $k\to\infty$}.
	\end{split}
\end{equation}

The main proofs of this paper are based on the analysis of the leading terms in \eqref{eq:asympGeneral}, or more generally, those in \eqref{eq:asympGen_diffeta}. For the appropriate non-circular domains we will show that they cannot be identically zero for all $\eta\in\Ss^1$, which in turn contradicts the existence of infinitely many non-scattering (Herglotz) wave numbers.

We finish this section by pointing out a simple fact that will be used later.
\begin{lem}\label{lem:Psi=0}
Given $\eta \in \Ss^1$ let $\theta,\vT_\xi$ satisfy $\partial\psi_{\eta}/\partial\theta=0$ as in \eqref{eq:stationaryGeneral}. Then $\Psi_{\eta}(\theta,\vT_\xi)=0$ if and only if $y'(\theta)=0$.
\end{lem}
\begin{proof} Notice from \eqref{eq:psi} and \eqref{eq:stationaryGeneral} that $\partial\psi_{\eta}/\partial\theta=\Psi_{\eta}=0$ if and only if $y'(\theta)=0$ or $(\sqrt{q}\,\eta+\xi)\cdot (\sqrt{q}\,\eta-\xi)=q-1=0$. The latter contradicts the assumption that $q\neq 1$. The proof is complete.
\end{proof}

\section{Ellipses}
Up to a rotational change of coordinates (which our results are invariant under) an ellipse $\Omega$, centered at the origin, is given by
\begin{equation}\label{eq:ellipse}
	\Omega=\Omega_{a,b}= \{ (x_1,x_2):~ (x_1/a)^2+ (x_2/b)^2<1\},
\end{equation}
with $0<b<a$.
We apply the parameterization $y=(y_1,y_2)^T= (a\cos\theta,b  \sin \theta)^T$ for $\partial\Omega$.
Denote $s:=b/a\in(0,1)$.
Then the conditions \eqref{eq:stationaryGeneral} for a stationary point $(\theta,\vT_{\xi})\in\cS=\cS_{\Omega,\eta}$ read
\begin{equation*}
	-(\sqrt{q}\,\eta_1+\xi_1)\sin \theta
	+s\, (\sqrt{q} \,\eta_2+\xi_2) \cos\theta =0
	\AND
	- \xi_2 \cos\theta+ s\, \xi_1 \sin \theta=0,
\end{equation*}
which, by substitution of the second identity into the first one, is equivalent to
\begin{equation}\label{eq:critEll}
	\sqrt{q}\,(\eta_1\,\xi_2-s^2\,\eta_2\,\xi_1)+(1-s^2)\xi_1\xi_2=0 \AND
	- \xi_2 \cos\theta+ s\, \xi_1 \sin \theta=0.
\end{equation}

The following result is proven in \cite[Lemma 7]{VogX20}.
\begin{lem}\label{lem:4critEll}
	If $s^2>1/(1+\sqrt{q})$, then for each $\eta\in \Ss^1$, there are exactly two solutions $\vT_\xi=\Cp_j\vT_\eta$, $j=1,2$, to the first equation of \eqref{eq:critEll}. These may be ordered to satisfy \begin{equation}\label{eq:critLocEll}
		\mbox{$\sgn\cos\vT_\eta=\sgn\cos\Cp_1\vT_\eta=-\sgn\cos\Cp_2\vT_\eta$,\qquad $\sgn\sin\vT_\eta=\sgn\sin\Cp_1\vT_\eta=-\sgn\sin\Cp_2\vT_\eta$.}
	\end{equation}
	With this ordering,
	\begin{equation}\label{eq:crit0piEll}
		\Cp_1\vT_\eta=\vT_\eta	\qquad\text{if and only if}\qquad \Cp_2\vT_\eta=\vT_{\eta}+\pi
		\qquad\text{if and only if}\qquad \vT_\eta\in\{0,\pi/2,\pi,3\pi/2\}.
	\end{equation}
	Moreover, we also have that $|\cos\vT_\eta| \leq\left|\cos \Cp_1\vT_\eta\right|$, and
	\begin{equation*}
		\begin{cases}
			|\cos\vT_\eta| \leq\left|\cos \Cp_2\vT_\eta\right| \leq|\cos\vT_\eta|/s^2&\mbox{if $q>1$},
			\\
			|\cos\vT_\eta|\ge \left|\cos \Cp_2\vT_\eta\right| &\mbox{if $0<q<1$},
		\end{cases}
	\end{equation*}
	where the leftmost equal signs hold if and only if $\vT_\eta\in\{0,\pi/2,\pi,3\pi/2\}$.
	As a consequence, we have in total four stationary points, i.e., solutions to \eqref{eq:critEll}, given by
	\begin{equation*}
		(\Cp_{j,s}\vT_\eta\,,\, \Cp_j\vT_\eta)
		\AnD(\Cp_{j,s}\vT_\eta+\pi\,,\, \Cp_j\vT_\eta),\qquad j=1,2,
	\end{equation*}
	where
	\begin{equation*}
		(\cos\Cp_{j,s}\vT_\eta\,,\,\sin\Cp_{j,s}\vT_\eta) =\frac{(s\cos\Cp_{j}\vT_\eta\,,\,\sin\Cp_{j}\vT_\eta)}{\sqrt{s^2\cos^2\Cp_{j}\vT_\eta+\sin^2\Cp_{j}\vT_\eta}}.
	\end{equation*}
\end{lem}
The following is a direct consequence of Lemma~\ref{lem:4critEll}.
\begin{cor}\label{cor:4critEll}
	Under the same assumptions and notations as in Lemma~\ref{lem:4critEll}, then
	\begin{equation*}
		\mbox{$\Cp_{j}(\vT_{\eta}+\pi)=\Cp_{j}\vT_{\eta}+\pi$, \qquad$j=1,2$.}
	\end{equation*}
	In addition, using $\xi_1=\cos\Cp_{j}\vT_{\eta}$ and $\xi_2=\sin\Cp_{j}\vT_{\eta}$, we get
	\begin{equation*}
		|\eta_1|/|\xi_1|-s^2|\eta_2|/|\xi_2|=(-1)^j(1-s^2)/\sqrt{q},\qquad \mbox{whenever $\eta_1\eta_2\neq 0$}.
	\end{equation*}
	As a consequence,
	\begin{equation*}
		\mbox{$|\cos \Cp_2\vT_\eta|<|\cos \Cp_1\vT_\eta|$\qquad when $\eta_1\eta_2\neq 0$.}
	\end{equation*}
\end{cor}
\begin{rem}\label{rem:4critEll}
		We note that \eqref{eq:critLocEll} implies that $\Cp_{1}\vT_\eta$ and $\vT_\eta$ lie in the same quadrant (or axis) in the $\RR^2$ plane, while $\Cp_{2}\vT_\eta$ and $\vT_\eta+\pi$ also lie in the same quadrant (or axis). Moreover, for $\vT_{\eta}\in(0,\pi/2)$,  the conclusion $|\cos\vT_\eta| \leq\left|\cos \Cp_1\vT_\eta\right|$ in Lemma~\ref{lem:4critEll} implies that  $\Cp_{1}\vT_{\eta}<\vT_{\eta}$. Additionally, since $|\cos \Cp_2\vT_\eta|<|\cos \Cp_1\vT_\eta|$ by Corollary \ref{cor:4critEll}, we also infer that $\Cp_{1}\vT_{\eta}<\Cp_{2}\vT_{\eta}-\pi$.
\end{rem}
\begin{lem}\label{lem:critEll}
	With the same assumptions and notations as in Lemma~\ref{lem:4critEll}, the mappings $\Cp_{j}:\mathbb{R}/(2\pi\ZZ)\to\mathbb{R}/(2\pi\ZZ)$, $j=1,2$, are $C^1$ bijections, and they are locally strictly  increasing\footnote{The monotonicity holds when we take the domain and the range of $\Cp_{j}$ to be fixed $2\pi$-length intervals. In other words, $\Cp_{j}:(t_0,t_0+2\pi)\to (\theta_0+2\kappa\pi,\theta_0+2(\kappa+1)\pi))$ is strictly increasing for any $t_0\in\RR$ and any $\kappa\in\ZZ$, where $\theta_0\cong \Cp_{j}t_0$ is the solution to the first equation of \eqref{eq:critEll} as specified in Lemma~\ref{lem:4critEll}.}.
\end{lem}
\begin{proof}
	It can be calculated straight forwardly from \eqref{eq:critEll} that
	\begin{equation} \label{eq:firstder}
		\frac{d\Cp_{j}\vT_{\eta}}{d\vT_{\eta}}=\frac{\sqrt{q}\,(s^2\,\eta_1\,\xi_1+\eta_2\,\xi_2)}{\sqrt{q}\,(\eta_1\,\xi_1+s^2\,\eta_2\,\xi_2)+(1-s^2)(\xi_1^2-\xi_2^2)},
	\end{equation}
	where $\xi_1=\cos\Cp_{j}\vT_{\eta}$ and $\xi_2=\sin\Cp_{j}\vT_{\eta}$. When
	$\vT_{\eta} \notin \{0,\pi/2,\pi,3\pi/2\}$, that is when $\eta_1\eta_2\ne 0$ ( and thus $\xi_1\xi_2\ne 0$), then \eqref{eq:firstder} may be rewritten
	\begin{equation} \label{eq:firstder2}
		\frac{d\Cp_{j}\vT_{\eta}}{d\vT_{\eta}}=
		\frac{s^2\,\eta_1\,\xi_1+\eta_2\,\xi_2}{\frac{\eta_1}{\xi_1}\,\xi_2^2+s^2\,\frac{\eta_2}{\xi_2}\,\xi_1^2},
		\qquad j=1,2.
	\end{equation}
	For $\vT_{\eta} \in \{0,\pi/2,\pi,3\pi/2\}$ \eqref{eq:firstder} gives
	\begin{equation*}
		\Cp_j'0=\Cp_j'\pi=\frac{s^2\sqrt{q}}{\sqrt{q}+(-1)^{j-1}(1-s^2)}
		\AND
		\Cp_j'\frac{\pi}{2}=\Cp_j'\frac{3\pi}{2}=\frac{\sqrt{q}}{s^2\sqrt{q}-(-1)^{j-1}(1-s^2)}.
	\end{equation*}
	Applying \eqref{eq:critLocEll} and the condition that $s^2>1/(1+\sqrt{q})$ we can deduce from \eqref{eq:firstder} and \eqref{eq:firstder2} that
	\begin{equation*}
		{d\Cp_{j}\vT_{\eta}}/{d\vT_{\eta}}>0,\qquad\mbox{for all $\eta\in\Ss^1$ and both $j=1,2$.}
	\end{equation*}
	Due to the monotonicity, the surjectivity of $\Cp_j$ can be obtained from \eqref{eq:crit0piEll}.
\end{proof}

We are now ready to prove Theorem~\ref{thm:ellipse}.
\begin{proof}[Proof of Theorem~\ref{thm:ellipse}]
We prove the result by contradiction. Let $\phi$ be a $C^1$ function which is not identically zero. Assume that there are infinitely many wave numbers $k_n$, $n\in\NN$ such that \eqref{eq:OverDet} admits a solution $u_{k_n}$. Then, as explained before, $k_n^2$ are transmission eigenvalues (for \eqref{eq:ITEP}) and so  $k_n$ must accumulate at $\infty$.
Since  $e^2=1-s^2<\sqrt{q}/(1+\sqrt{q})$ implies that $s^2>1/(1+\sqrt{q})$  lemmata~\ref{lem:4critEll} ~and~\ref{lem:critEll} and Corollary~\ref{cor:4critEll} are all valid. Since $y(t+\pi)=-y(t)$, it follows immediately from the definition of the functions $\psi_\eta$ and $\Psi_\eta$ in \eqref{eq:psi} that
\begin{equation*}
	\{\psi_\eta,\Psi_\eta,D^2\psi_\eta\}(\theta+\pi,\vT_{\xi})=-\{\psi_\eta,\Psi_\eta,D^2\psi_\eta\}(\theta,\vT_{\xi}).
\end{equation*}
Consider the asymptotic formulas \eqref{eq:asympGen_diffeta} and recall from Lemma~\ref{lem:4critEll} that there are exactly four stationary points given by $(\theta,\theta_{\xi})=(\Cp_{j,s}\vT_\eta\,,\, \Cp_j\vT_\eta)$ and $(\theta,\theta_{\xi})=(\Cp_{j,s}\vT_\eta+\pi\,,\, \Cp_j\vT_\eta)$, $j=1,2$. We then obtain that, as $k=k_n\to\infty$,
\begin{equation}\label{eq:asympEll}
	\begin{split}
		&
		\sum_{j=1,2} \text{\small$\frac{\phi(\Cp_{j}\vT_\eta)\,\Psi_j(\vT_\eta)}{\abs{\det D^2\psi_j(\vT_\eta)}^{1/2}}$}
		\left[e^{\frac{i\pi}4\sgn D^2\psi_j(\vT_\eta)+ik\psi_j(\vT_\eta)}f_{\eta,j}^N- 		e^{-\frac{i\pi}4\sgn D^2\psi_j(\vT_\eta)-ik\psi_j(\vT_\eta)}(-f_{\eta,j})^N\right]\to 0,
	\end{split}
\end{equation}
for all $\eta\in\Ss^1$, where
\begin{equation*}
	\{\psi_j(\vT_\eta),\Psi_j(\vT_\eta),D^2\psi_j(\vT_\eta)\}=\{\psi_\eta,\Psi_\eta,D^2\psi_\eta\}(\Cp_{j,s}\vT_\eta,\Cp_j\vT_\eta),
\end{equation*}
and (originating from the $\eta^\perp\cdot y$ term)
\begin{equation*}
	f_{\eta,j}=\frac{b \sin(\Cp_j\vT_{\eta}-\vT_\eta)}{\sqrt{s^2\cos^2\Cp_{j}\vT_\eta+\sin^2\Cp_{j}\vT_\eta}}.
\end{equation*}
Here we have made use of the fact that $\det D^2\psi_j(\vT_\eta)\neq 0$ for $j=1,2$ and $\eta\in\Ss^1$ (see \cite[Lemma 8]{VogX20}). By extraction of a subsequence (depending on $\vT_\eta$) we may assume that
$$
e^{ik\psi_j(\vT_\eta)} \rightarrow z_j(\vT_\eta) \hbox { as } k=k_n \rightarrow \infty
$$
for $j=1,2$ where $|z_j(\vT_\eta)|=1$. From \eqref{eq:asympEll} with $0\le N\le 3$, we now conclude that
\begin{equation*}
		\begin{bmatrix}
			1&1&1&1\\a_1&a_2&a_3&a_4\\a_1^2&a_2^2&a_3^2&a_4^2\\a_1^3&a_2^3&a_3^3&a_4^3
		\end{bmatrix}
		\begin{bmatrix}
			t_1\\t_2\\t_3\\t_4
		\end{bmatrix}=0
		\end{equation*}
Here $a_1=-a_2=f_{\eta,1}$, $a_3=-a_4=f_{\eta,2}$ and
\begin{eqnarray*}
t_1=\frac{\phi(\Cp_{1}\vT_\eta)\,\Psi_1(\vT_\eta)}{\abs{\det D^2\psi_1(\vT_\eta)}^{1/2}}e^{\frac{i\pi}4\sgn D^2\psi_1(\vT_\eta)}z_1~, ~~~ t_2=-\frac{\phi(\Cp_{1}\vT_\eta)\,\Psi_1(\vT_\eta)}{\abs{\det D^2\psi_1(\vT_\eta)}^{1/2}}e^{\frac{-i\pi}4\sgn D^2\psi_1(\vT_\eta)}z_1^{-1} \\
t_3=\frac{\phi(\Cp_{2}\vT_\eta)\,\Psi_2(\vT_\eta)}{\abs{\det D^2\psi_2(\vT_\eta)}^{1/2}}e^{\frac{i\pi}4\sgn D^2\psi_2(\vT_\eta)}z_2~,~~~ t_4=-\frac{\phi(\Cp_{2}\vT_\eta)\,\Psi_2(\vT_\eta)}{\abs{\det D^2\psi_2(\vT_\eta)}^{1/2}}e^{-\frac{i\pi}4\sgn D^2\psi_2(\vT_\eta)}z_2^{-1}~.
\end{eqnarray*}
 We notice that $f_{\eta,j}=0$ if and only if $\eta_1\eta_2=0$. When $\eta_1 \eta_2\ne 0$ and $|f_{\eta,1}|\ne |f_{\eta,2}|$ the above Vandermonde matrix is invertible, and thus $t_1=t_2=t_3=t_4=0$, in other words one possibility for $\eta_1\eta_2\ne 0$ is
\begin{equation}\label{eq:proof0Ell}
|f_{\eta,1}|\ne |f_{\eta,2}| \AND	\phi(\Cp_{1}\vT_\eta)\Psi_1(\vT_\eta)=\phi(\Cp_{2}\vT_\eta)\Psi_2(\vT_\eta)=0,
\end{equation}
When $|f_{\eta,1}|= |f_{\eta,2}|$	and $\eta_1\eta_2\ne0$ the above Vandermonde matrix is not invertible (has rank 2) and we get $t_1+t_3=t_2+t_4=0$ or $t_1+t_4=t_2+t_3=0$. In either case the alternative to \eqref{eq:proof0Ell} for $\eta_1\eta_2\ne 0$ becomes
\begin{equation}\label{eq:prooff}
	|f_{\eta,1}|=|f_{\eta,2}|
	\AND
	\frac{|\phi(\Cp_{1}\vT_\eta)|\,|\Psi_1(\vT_\eta)|}{\abs{\det D^2\psi_1(\vT_\eta)}^{1/2}}
	=\frac{|\phi(\Cp_{2}\vT_\eta)|\,|\Psi_2(\vT_\eta)|}{\abs{\det D^2\psi_2(\vT_\eta)}^{1/2}}.
\end{equation}
The last statement can be rearranged as $|f_{\eta,1}|=|f_{\eta,2}|$ and
\begin{equation}\label{eq:proofIdEll1}
	|\phi(\Cp_{2}\vT_\eta)|^2=\Gp(\vT_{\eta})|\phi(\Cp_{1}\vT_\eta)|^2\qquad\text{with}\quad
	\Gp(\vT_{\eta})=\frac{\abs{\det D^2\psi_2(\vT_\eta)}}{\abs{\det D^2\psi_1(\vT_\eta)}}\frac{\Psi_1(\vT_\eta)^2}{\Psi_2(\vT_\eta)^2}.
\end{equation}

Here we used that $\Psi_j(\vT_\eta)\neq0$ for any $\vT_{\eta}$. This follows immediately from Lemma~\ref{lem:Psi=0}  since $y'(\theta)=a(-\sin\theta,s\cos\theta)\neq 0$ for any $\theta$.
Based on \eqref{eq:proof0Ell} and \eqref{eq:proofIdEll1}, and the fact that $\Psi_j(\vT_\eta)\neq0$ for any $\vT_{\eta}$, we conclude that in both cases, whether $|f_{\eta,1}|\neq |f_{\eta,2}|$ or $|f_{\eta,1}|= |f_{\eta,2}|$, one has
\begin{equation}\label{eq:proofIdEll}
	|\phi(\Cp_{2}\vT_\eta)|^2=\Gp(\vT_{\eta})|\phi(\Cp_{1}\vT_\eta)|^2\qquad\text{with}\quad
	\Gp(\vT_{\eta})=\frac{\abs{\det D^2\psi_2(\vT_\eta)}}{\abs{\det D^2\psi_1(\vT_\eta)}}\frac{\Psi_1(\vT_\eta)^2}{\Psi_2(\vT_\eta)^2}.
\end{equation}

Moreover, due to the continuity of $\phi$, $\Gp$, and $\Cp_j$, we infer that \eqref{eq:proofIdEll} holds for all $\vT_\eta$, and not just when $\eta_1\eta_2\ne 0$. The argument used to arrive at \eqref{eq:proof0Ell} \eqref{eq:proofIdEll1} will be developed further in Section \ref{sec:starshape} where we study more general star-shaped domains (see Lemma~\ref{lem:asympStarBig} and Corollary~\ref{cor:phiStarBig})).

We obtain by direct calculations that
\begin{equation*}
	\Gp(0)=\Gp(\pi)=\frac{\sqrt{q}-1+s^2}{\sqrt{q}+1-s^2}\frac{(\sqrt{q}-1)^2}{(\sqrt{q}+1)^2}<1.
\end{equation*}
Then by continuity we have that
\begin{equation}\label{eq:G<1}
	\Gp(\vT_{\eta})<1,\qquad \mbox{for $j=1,2$ and $\vT_\eta \in (-t_1,t_1)\cup (\pi -t_1,\pi +t_1)$},
\end{equation}
with some $t_1\in(0,\pi/2)$.
In addition, we obtain from \eqref{eq:proofIdEll} with $\vT_{\eta}=0,\pi$ that $\phi(0)=\phi(\pi)=0$; otherwise
	$$|\phi(\pi)|^2=\Gp(0)|\phi(0)|^2=\Gp(0)\Gp(\pi)|\phi(\pi)|^2<|\phi(\pi)|^2,$$
and similarly for $\phi(0)$, which is a contradiction.
From Lemma~\ref{lem:4critEll} and Corollary~\ref{cor:4critEll} we get
\begin{equation}\label{eq:Cpjmap}
	\mbox{$\Cp_{1}(0,\pi/2)=\Cp_{2}(\pi,3\pi/2)=(0,\pi/2)$,
		\qquad
		$\Cp_{1}(\pi,3\pi/2)=\Cp_{2}(0,\pi/2)=(\pi,3\pi/2)$,}\footnote{We note here that ``$(0,\pi/2)$'' denotes the interval from $0$ to $\pi/2$, and not a coordinate.}
\end{equation}
and (from Remark~\ref{rem:4critEll})
\begin{equation}\label{eq:Cpjmono}
	\Cp_{1}t~<~t\AnD\Cp_{1}(t+\pi)-\pi~=~\Cp_{1}t~<~\Cp_{2}t-\pi~=~\Cp_2(t+\pi),\qquad \mbox{for all $t\in(0,\pi/2)$}.
\end{equation}

Next we show that $\phi=0$ in $(0,\Cp_1t_1)\subset(0,t_1)\subset(0,\pi/2)$. If not, then we can find $t_2\in(0,t_1]$ such that
\begin{equation}\label{eq:proofNormEll}
	|\phi(\Cp_1t_2)|=\|\phi\|_{C^0[0,\Cp_1t_1]}>0\AND |\phi(t)|<|\phi(\Cp_1t_2)|\quad\mbox{for all $t\in(0,\Cp_1t_2)$}.
\end{equation}
Denote $\Cp_0=\Cp_2^{-1}\Cp_1$. We observe from \eqref{eq:Cpjmap}, \eqref{eq:Cpjmono} and Lemma~\ref{lem:critEll} that $\Cp_0: (0,\pi/2) \to (\pi,3\pi/2)$ and $\Cp_0: (\pi,3\pi/2)\to (0,\pi/2)$ are both increasing bijections. Moreover,
\begin{equation}\label{eq:proofMonoEll}
	\pi<\Cp_0t<t+\pi \AND 0<\Cp_0^2t<\Cp_0(t+\pi)<t, 
	\qquad \mbox{for all $t\in(0,\pi/2)$}.
\end{equation}
Then we can apply \eqref{eq:proofIdEll} and \eqref{eq:G<1} for $\vT_{\eta}=\Cp_2^{-1}\Cp_1t_2=\Cp_0t_2\in(\pi,\pi+t_2)$ and for $\vT_{\eta}=\Cp_0^2t_2\in(0,t_2)$, which yields
\begin{equation*}
	|\phi(\Cp_1t_2)|^2=\Gp(\Cp_0t_2)|\phi(\Cp_{1}\Cp_0t_2)|^2=\Gp(\Cp_0t_2)\Gp(\Cp_0^2t_2)|\phi(\Cp_{1}\Cp_0^2t_2)|^2<|\phi(\Cp_{1}\Cp_0^2t_2)|^2.
\end{equation*}
However, it contradicts \eqref{eq:proofNormEll} since $0<\Cp_{1}\Cp_0^2t_2<\Cp_1t_2$ by \eqref{eq:proofMonoEll}.

We can now prove $\phi=0$ in $(0,\pi/2)$.   Since $\phi=0$ in $(0,\Cp_1t_1)$, we obtain from \eqref{eq:proofIdEll} that $\phi=0$ in $(\pi,\Cp_{2}t_1)$. We can apply \eqref{eq:proofIdEll} for $\vT_\eta\in(\pi,\Cp_{1}^{-1}\Cp_{2}t_1)$ to obtain $\phi=0$ in $(0,\Cp_{2}\Cp_{1}^{-1}\Cp_{2}t_1)= (0,\Cp_{1}\CP_{0}^2 t_1)$, where $\CP_{0}:=\Cp_{1}^{-1}\Cp_{2}$. With an argument of induction, we can now show that
\begin{equation*}
	\phi=0 \qquad\mbox{in $(0,\Cp_1\CP_{0}^{2m}t_1)$, \quad for any $m\in\mathbb{N}$.}
\end{equation*}

From \eqref{eq:proofMonoEll} it follows that that $t<\CP_{0}^2t<\pi/2$ for all $t\in(0,\pi/2)$. As a consequence, $\{\Cp_1\CP_{0}^{2m}t_1\}_{m\in\mathbb{N}}$ is a strictly increasing sequence in $(0,\pi/2)$. Therefore it has a limit in $(0,\pi/2]$ as $m\to\infty$, denoted as $t_0$. In fact we must have $t_0=\pi/2$; otherwise $\Cp_1^{-1}t_0=\CP_{0}^2 \Cp_1^{-1}t_0>\Cp_1^{-1}t_0$. Therefore, we have deduced that $\phi=0$ in $(0,\pi/2)$.

Applying  \eqref{eq:proof0Ell} or \eqref{eq:proofIdEll} again for $\vT_{\eta}\in(0,\pi/2)$ yields $\phi=0$ in $(\pi,3\pi/2)$. Using analogous arguments we can also deduce that $\phi=0$ in $(\pi/2,\pi)\cup(3\pi/2,2\pi)$. In particular, we need to make use of the properties
\begin{equation*}
	\mbox{$\Cp_{1}(\pi/2,\pi)=\Cp_{2}(3\pi/2,2\pi)=(\pi/2,\pi)$,
		\qquad
		$\Cp_{1}(3\pi/2,2\pi)=\Cp_{2}(\pi/2,\pi)=(3\pi/2,2\pi)$,}
\end{equation*}
and
\begin{equation*}
	\Cp_{1}t~>~t\AnD\Cp_{1}(t+\pi)-\pi~=~\Cp_{1}t~>~\Cp_{2}t-\pi~=~\Cp_2(t+\pi),\qquad \mbox{for all $t\in(\pi/2,\pi)$},
\end{equation*}
which can be obtained from Lemma~\ref{lem:4critEll} and Corollary~\ref{cor:4critEll}.
\end{proof}

\section{Star-Shaped Domains}\label{sec:starshape}
Let $\Omega$ be a $C^2$ domain that is star-shaped with respect to the origin and with a radius function $\rho$. In other words $\partial\Omega=\{y(\theta):=\rho(\theta)\vec{\theta}; \theta\in[0,2\pi)\}$. Then
\begin{equation}\label{eq:y''}
	y'=(\rn'\vec{\theta}+\vec{\theta}^{\perp})\rho\AnD
	y''=\Pare{(\rn'^2-1+\rn'')\vec{\theta}+2\rn'\vec{\theta}^{\perp}}\rho,
\qquad\mbox{where $\rn:=\ln\rho$}.
\end{equation}
Recall that $(\theta,\vT_{\xi})$ belongs to the set of stationary points $\cS=\cS_{\Omega,\eta}$ (of $\psi_\eta$) if and only if \eqref{eq:stationaryGeneral2} is satisfied, or equivalently, for $l=1$ or $2$,
\begin{equation}\label{eq:criticalStar}
	\xi=(-1)^{l-1}\,\vec{\theta}
	\AnD
	\Pare{\vec{\theta}\cdot\eta_l\sqrt{q}+1}\rn'(\theta)=-\vec{\theta}^{\perp}\cdot\eta_l\sqrt{q}
	,\qquad\mbox{with $\eta_l:=(-1)^{l-1}\eta$}.
\end{equation}
Notice that $\theta$ that satisfies $\vec{\theta}\cdot\eta_l\sqrt{q}+1=0$ can never be a solution to \eqref{eq:criticalStar}, since $q\ne 1$. Therefore the second condition in \eqref{eq:criticalStar} for stationary points can equivalently be written as
\begin{equation}\label{eq:criticalStar2}
	\rn'(\theta)=h(\theta-\vT_{\eta_l})=\frac{\sqrt{q}\sin(\theta-\vT_{\eta_l})}{\sqrt{q}\cos(\theta-\vT_{\eta_l})+1},
\end{equation}
where $h$ is given by
\begin{equation}\label{eq:hstar}
	h(\theta)=\frac{\sqrt{q}\sin\theta}{\sqrt{q}\cos\theta+1}.
\end{equation}
Applying \eqref{eq:criticalStar} we have for every $(\theta,\vT_{\xi})\in \cS_{\Omega,\eta}$ with $\xi=(-1)^{l-1}\vec{\theta}$ that
\begin{equation}\label{eq:y'calc}
	-\frac{y'(\theta)^\perp}{\rho(\theta)}
	=\vec{\theta}- \vec{\theta}^\perp\rn'(\theta)
	=\frac{(\vec{\theta}\cdot\eta_l\sqrt{q}+1)\vec{\theta} +\pare{\vec{\theta}^\perp\cdot\eta_l}\vec{\theta}^\perp\sqrt{q}}{\vec{\theta}\cdot\eta_l\sqrt{q}+1}
	=\frac{\eta_l\sqrt{q} +\vec{\theta} }{\vec{\theta}\cdot\eta_l\sqrt{q}+1},
\end{equation}
where in the last equality we have made use of the identity $\eta=(\vec{\theta}\cdot\eta)\vec{\theta}+(\vec{\theta}^\perp\cdot\eta)\vec{\theta}^\perp$ with $\eta=\eta_l$. Inserting \eqref{eq:y''} and \eqref{eq:y'calc} into \eqref{eq:psi} and \eqref{eq:Hessian} we obtain that
\begin{equation}\label{eq:PsiStationaryStar}
	\frac{\Psi_\eta(\theta,\vT_{\xi})}{\rho(\theta)}
	=\frac{(-1)^{l-1}(q-1)}{\vec{\theta}\cdot\eta_l\sqrt{q}+1}
	=\frac{(q-1)\rho(\theta)}{\psi_\eta(\theta,\vT_{\xi})},
\end{equation}
\begin{equation}\label{eq:HessianStar}
	\begin{split}
		(D^2\psi_{\eta})(\theta,\vT_{\xi})
		&=
		-(-1)^{l-1}\rho(\theta) \begin{bmatrix}
			\Pare{1+\rn'^2-\rn''}|_\theta(\vec{\theta}\cdot\eta_l\sqrt{q}+1)&-1
			\\-1&1
		\end{bmatrix}.
	\end{split}
\end{equation}

The proofs of Theorems~\ref{thm:Star1},~\ref{thm:main1Star}, and~\ref{thm:StarAnly} follow similar ideas as in the proof for the case of an ellipse. However, since we do not have the explicit structure of $\rho$ (or even the $\pi$-periodicity), a more detailed analysis of the stationary points is required. In the rest of this section, we shall first prove the existence of stationary points as well as their properties as a function of $\eta\in\Ss^1$. Then we establish the precise form of the leading term in the asymptotics \eqref{eq:asympGen_diffeta}, and discuss its implications using a Vandermonde matrix similar to that in the proof of Theorem~\ref{thm:ellipse}. Finally, we prove the results related to star-shaped domains.
\subsection{The Stationary Points for $q>1$}
\label{sec:statStarBig}
We consider the case when $q>1$ and $\rho$ satisfies \eqref{eq:StarBigCond1}.
Recall the function $h$ as defined in \eqref{eq:hstar}.
Then
\begin{equation}\label{eq:hdiffstar}
	h'(\theta)
	~=~
	\frac{\sqrt{q}(\sqrt{q}+\cos\theta)}{\pare{\sqrt{q}\cos\theta+1}^2}
	~\ge~	\frac{\sqrt{q}}{1+\sqrt{q}}>0\qquad\mbox{for all $\theta$}.
\end{equation}
Hence $h$ has the range $(-\infty,\infty)$, and monotonically increases on both the intervals $(\theta_q-\pi,\pi-\theta_q)$ and $(\pi-\theta_q,\theta_q+\pi)$, where $\theta_q=\arccos (1/\sqrt{q})\in(0,\pi/2)$.
\begin{figure}[h]
	\begin{subfigure}{.55\textwidth}
		\centering\includegraphics{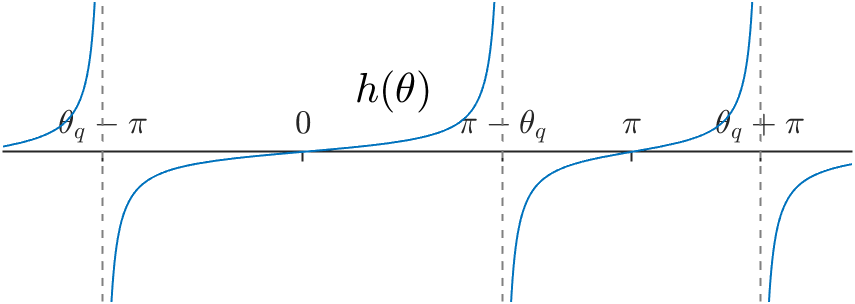}
	\end{subfigure}
	\begin{subfigure}{.45\textwidth}
		\centering
		\includegraphics{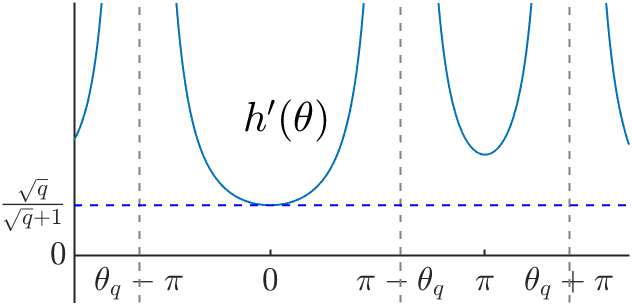}
	\end{subfigure}
	\caption{Graphs of $h$ and $h'$ with $q=9$.}
\end{figure}
Therefore, we have
\begin{lem}\label{lem:4criticStarBig}
	Suppose that $q>1$ and $\rho$ is a $C^2$ function satisfying \eqref{eq:StarBigCond1}.
	Then for each $\eta\in\Ss^1$ and each $l\in\{1,2\}$, there are exactly two solutions $\Cp_{1,l}\eta$ and $\Cp_{2,l}\eta$ to \eqref{eq:criticalStar2}, which satisfy $\Cp_{1,l}\eta-\vT_{\eta_l} \in(\theta_q-\pi, \pi-\theta_q)$ and $\Cp_{2,l}\eta-\vT_{\eta_l} \in (\pi-\theta_q, \theta_q+\pi)$.
\end{lem}
We note that\footnote{For indices in $\{1,2\}$, addition is performed modulo $2$.}
\begin{equation}\label{eq:CppiStarBig}
	\Cp_{j,l}\tilde{\eta}=\Cp_{j,l+1}\eta, \qquad\mbox{for $\tilde{\eta}=-\eta$, $\eta\in\Ss^1$, and $j,l=1,2$,}
\end{equation}
and
\begin{equation}\label{eq:cossintStarBig}
	(-1)^{j-1}\pare{\sqrt{q}\cos(\Cp_{j,l}\eta-\vT_{\eta_l})+1}>0
	\AnD \sgn\sin(\Cp_{j,l}\eta-\vT_{\eta_l})=(-1)^{j-1}\sgn\rho'(\Cp_{j,l}\eta).
\end{equation}
Notice that \eqref{eq:criticalStar2} implies $q(g^2+1)t^2+2\sqrt{q}g^2t+g^2-q=0$, where $t=\cos(\theta-\vT_{\eta_l})$ and $g=\rn'(\theta)$.
Thanks to \eqref{eq:cossintStarBig} we derive that
\begin{equation}\label{eq:cost2StarBig}
	\sqrt{q}\cos(\Cp_{j,l}\eta-\vT_{\eta_l})
	=
	\frac{-\rn'^2-(-1)^j\sqrt{q+(q-1)\rn'^2}}{1+\rn'^2}\,\Big|_{\Cp_{j,l}\eta},
\end{equation}
and
\begin{equation}\label{eq:sint2StarBig}
	\sqrt{q}\sin(\Cp_{j,l}\eta-\vT_{\eta_l})
	=\rn'\frac{1-(-1)^j\sqrt{q+(q-1)\rn'^2}}{1+\rn'^2}\,\Big|_{\Cp_{j,l}\eta}.
\end{equation}
As a consequence, for each $j,l=1,2$ we have
\begin{equation}\label{eq:critDiffrStarBig}
	\Cp_{j,l-1+j}\eta=\vT_{\eta_l}~\qquad\mbox{if and only if}\qquad
	\rho'(\vT_{\eta_l})=0,
\end{equation}
by recalling \eqref{eq:CppiStarBig} and noticing that $q+(q-1)\rn'^2>1$.
We can regard the solutions $\Cp_{j,l}\eta$, $j=1,2$, to \eqref{eq:criticalStar2} as maps $\Cp_{j,l}:\Ss^1\to\mathbb{R}/\{2\pi\}$. When there is no ambiguity, we shall sometimes consider the domain of $\Cp_{j,l}$ as $\mathbb{R}/\{2\pi\}$ by identifying $\vT_{\eta_l}$ with $\eta_l$, or refer to the range fo $\Cp_{j,l}$ as $\Ss^1$ by identifying $\vec{\theta}$ with $\theta$.
\begin{lem}\label{lem:critMapStarBig}
	Under the same assumptions and notations as in Lemma~\ref{lem:4criticStarBig}, for each $j,l=1,2$, $\Cp_{j,l}:\Ss^1\to\Ss^1$ is a bijection and $\Cp_{j,l}:\mathbb{R}/(2\pi\ZZ)\to\mathbb{R}/(2\pi\ZZ)$ is $C^1$ and locally strictly  increasing\footnote{The monotonicity holds when we take the domain and the range of $\Cp_{j,l}$ to be fixed $2\pi$-length intervals. In other words, $\Cp_{j,l}:(t_0,t_0+2\pi)\to (\theta_0+2\kappa\pi,\theta_0+2(\kappa+1)\pi))$ is strictly increasing for any $t_0\in\RR$ and any $\kappa\in\ZZ$, where $\theta_0\cong \Cp_{j,l}t_0$ is the solution to \eqref{eq:criticalStar2} as specified in Lemma~\ref{lem:4criticStarBig}. The same applies to the regularity.}.
\end{lem}
\begin{proof}
	Given $j,l=1,2$, we first deduce from \eqref{eq:hdiffstar} and \eqref{eq:cost2StarBig} that
	\begin{equation}\label{eq:h'}
		(\sqrt{q}\cos t+1)h'(t)=(-1)^{j-1}\sqrt{q+(q-1)\rn'^2(\theta)},
	\end{equation}
	where $\theta=\theta(\eta)=\Cp_{j,l}\eta$ and $t=\theta-\vT_{\eta_l}$.
	Differentiating \eqref{eq:criticalStar2} with respect to $\vT_{\eta}$ and
applying \eqref{eq:hdiffstar} we arrive at
	\begin{equation*}
	\rn''(\theta)\theta'=\sqrt{q}(\theta'-1)\frac{\sqrt{q}+\cos(\theta-\vT_{\eta_l})}{\pare{\sqrt{q}\cos(\theta-\vT_{\eta_l})+1}^2}
	=(\theta'-1)h'(\theta-\vT_{\eta_l}),
	\end{equation*}
	where $\theta':=\partial\theta/\partial\vT_{\eta}=\partial\Cp_{j,l}\eta/\partial\vT_{\eta}$. Hence by \eqref{eq:h'} we have
	\begin{equation}\label{eq:dTstarbig}
		\dd{\theta}{\vT_{\eta}}
		=\frac{h'(t)}{h'(t)-\rn''(\theta)}
		=\dfrac{\sqrt{q+(q-1)\rn'^2(\theta)}}{\sqrt{q+(q-1)\rn'^2(\theta)}+(-1)^{j}\pare{\sqrt{q}\cos t+1}\rn''(\theta)}.
	\end{equation}
From \eqref{eq:cossintStarBig} it follows that $0>(-1)^{j}\pare{\sqrt{q}\cos t+1}>-(\sqrt{q}+1)$. Then, recalling the admissibility condition \eqref{eq:StarBigCond1} on $\rho$, we deduce
\begin{equation*}
	\sqrt{q+(q-1)\rn'^2(\theta)}+(-1)^{j}\pare{\sqrt{q}\cos t+1}\rn''(\theta)
	>\sqrt{q+(q-1)\rn'^2(\theta)}-\sqrt{q}\ge 0.
\end{equation*}
Consequently, $\partial\theta/\partial\vT_{\eta}>0$, and
	 $\Cp_{j,l}\eta$, $j,l=1,2$, are strictly increasing $C^1$ functions of $\vT_{\eta}$.
	Thanks to the monotonicity, the bijectivity will follow if $\Cp_{j,l}$ is surjective on $\Ss^1$.
	In fact given $j,l=1,2$, and $\vec{\theta}\in\Ss^1$, the equations \eqref{eq:cost2StarBig} and \eqref{eq:sint2StarBig}, with $\Cp_{j,l}\eta$ replaced by $\theta$, uniquely determine a  $\vT_{\eta_l}\in\RR/\{2\pi\}$ and hence a unique $\eta\in\Ss^1$. The proof is complete.
\end{proof}
\begin{lem}\label{lem:crimoreStarBig}
	Under the same assumptions and notations as in Lemma~\ref{lem:4criticStarBig}, we have for each $\eta\in\Ss^1$ and each $l=1,2$ that
	\begin{equation*}
		\sgn \rho'(\Cp_{1,l}\eta)=(-1)^{l-1}\sgn \sin(\Cp_{1,l}\eta-\vT_{\eta}) =(-1)^{l-1}\sgn \sin(\Cp_{2,l+1}\eta-\vT_{\eta})=\sgn\rho'(\Cp_{2,l+1}\eta).
	\end{equation*}
	Moreover, for each $l=1,2$ and every $\eta\in\Ss^1$ such that $\rho'(\vT_{\eta_l})\neq0$, we must have $0<|\Cp_{2,l+1}\eta-\vT_{\eta_l}|<|\Cp_{1,l}\eta-\vT_{\eta_l}|<\pi-\theta_q$, and
	\begin{equation*}
		\rho(\Cp_{2,l+1}\eta)<\rho(\Cp_{1,l}\eta)
		\AND 0<|\sin(\Cp_{2,l+1}\eta-\vT_{\eta})|<|\sin(\Cp_{1,l}\eta-\vT_{\eta})|.
	\end{equation*}
\end{lem}
\begin{proof}
	We first observe that $\theta=\Cp_{1,1}\eta$ is the unique solution to $\rn'(\theta)=h(\theta-\vT_{\eta})$ with $\theta-\vT_{\eta}\in(\theta_q-\pi, \pi-\theta_q)$, and $\theta=\Cp_{2,2}\eta$ is the unique solution to $\rn'(\theta)=h_\pi(\theta-\vT_{\eta})$ with $\theta-\vT_{\eta}\in(-\theta_q, \theta_q)\subset(\theta_q-\pi, \pi-\theta_q)$.
	Here, $h_\pi(\theta)=h(\theta-\pi)$.
	\begin{figure}
		\centering\includegraphics{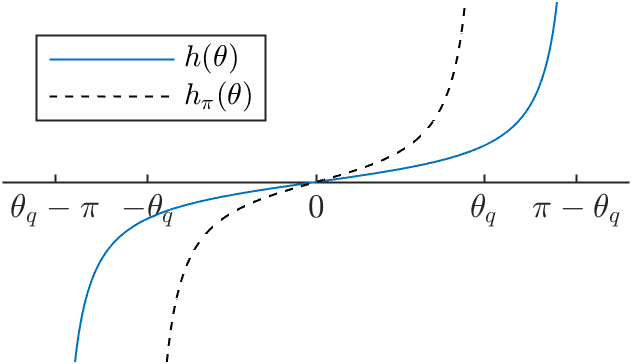}
		\caption{Graphs of $h$ and $h_\pi$ when $q=9$.}
	\end{figure}
	It can be verified straightforwardly that $h_{\pi}$ monotonically increases in $(-\theta_q,\theta_q)$ and that
	\begin{equation*}
		0<|h(\theta)|<|h_\pi(\theta)|,\qquad\mbox{for all $\theta\in(-\theta_q,\theta_q)\backslash\{0\}$}.
	\end{equation*}

	\textbf{(i)} If $\rho'(\Cp_{2,2}\eta)=0$, then $h_\pi(\Cp_{2,2}\eta-\vT_{\eta})=0$ and hence $\Cp_{2,2}\eta=\vT_{\eta}$. Consequently $\rho'(\vT_{\eta})=0$ and $\Cp_{1,1}\eta=\vT_{\eta}$.

	\textbf{(ii)} In the case when $\rho'(\Cp_{2,2}\eta)>0$, we first observe from \eqref{eq:cossintStarBig}  that $\Cp_{2,2}\eta-\vT_\eta\in(0,\theta_q)$.  We claim that $0<\Cp_{2,2}\eta-\vT_\eta<\Cp_{1,1}\eta-\vT_\eta<\pi-\theta_q$, and hence $\rho'(\Cp_{1,1}\eta)>0$. Otherwise if $\theta_q-\pi<\Cp_{1,1}\eta-\vT_\eta\le \Cp_{2,2}\eta-\vT_\eta<\theta_q$, and then
	\begin{equation*}
		\rn'(\Cp_{2,2}\eta)-\rn'(\Cp_{1,1}\eta)=h_\pi(\Cp_{2,2}\eta-\vT_\eta)-h(\Cp_{1,1}\eta-\vT_\eta)> h(\Cp_{2,2}\eta-\vT_\eta)-h(\Cp_{1,1}\eta-\vT_\eta)\ge 0.
	\end{equation*}
	Hence $\Cp_{1,1}\eta\neq\Cp_{2,2}\eta$ and, by the mean value theorem again, there exist $t_1,t_2\in (\Cp_{1,1}\eta, \Cp_{2,2}\eta)$ such that
	\begin{equation*}
		\rn''(t_1)>h'(t_2-\vT_\eta)\ge \sqrt{q}/(1+\sqrt{q}),
	\end{equation*}
	which contradicts \eqref{eq:StarBigCond1}.
	Now $0<\sin(\Cp_{2,2}\eta-\vT_\eta)<\sin(\Cp_{1,1}\eta-\vT_\eta)$ follows by noticing that either $0<\Cp_{2,2}\eta-\vT_\eta<\Cp_{1,1}\eta-\vT_\eta\le\pi/2$ or $0<\Cp_{2,2}\eta-\vT_\eta<\theta_q<\pi/2<\Cp_{1,1}\eta-\vT_\eta<\pi-\theta_q$ holds.
	Next, assume that $\rho(\Cp_{2,2}\eta)\ge\rho(\Cp_{1,1}\eta)$, and hence there exists $t_1\in(\Cp_{2,2}\eta, \Cp_{1,1}\eta)$ such that $\rho'(t_1)\le0$. Thus

	\begin{equation*}
		\rn'(\Cp_{1,1}\eta)-\rn'(t_1)\ge h(\Cp_{1,1}\eta-\vT_{\eta})>h(\Cp_{1,1}\eta-\vT_{\eta})-h(t_1-\vT_\eta),
	\end{equation*}
	where in the latter inequality we utilized the properties $t_1-\vT_{\eta}\in(\Cp_{1,1}\eta-\vT_{\eta}, \Cp_{2,2}\eta-\vT_{\eta})\subset(0,\pi-\vT_{\eta})$ and $h>0$ on $(0,\pi-\vT_{\eta})$. Consequently there exist $t_2,t_3\in (t_1,\Cp_{1,1}\eta)$ such that
	\begin{equation*}
		\rn''(t_2)>h'(t_3-\vT_\eta)\ge \sqrt{q}/(1+\sqrt{q}),
	\end{equation*}
	which again contradicts \eqref{eq:StarBigCond1}. Therefore $\rho(\Cp_{2,2}\eta)<\rho(\Cp_{1,1}\eta)$.

	\textbf{(iii)} By analogous arguments we can show that $\rho'(\Cp_{2,2}\eta)<0$ implies $\Cp_{2,2}\eta-\vT_\eta\in(-\theta_q,0)$, $0>\Cp_{2,2}\eta-\vT_\eta>\Cp_{1,1}\eta-\vT_\eta>\theta_q-\pi$, $\rho'(\Cp_{1,1}\eta)<0$, $0>\sin(\Cp_{2,2}\eta-\vT_\eta)>\sin(\Cp_{1,1}\eta-\vT_\eta)$, and $\rho(\Cp_{2,2}\eta)<\rho(\Cp_{1,1}\eta)$.

	This verifies the statements of the lemma for $l=1$, the statements for $l=2$ can be proven analogously by identifying $\Cp_{1,2}\eta$ as the unique solution to $\rn'(\theta)=h(\theta-\vT_{\eta_2})$ with $\theta-\vT_{\eta_2}\in(\theta_q-\pi, \pi-\theta_q)$, and $\Cp_{2,1}\eta$ the unique solution to $\rn'(\theta)=h_\pi(\theta-\vT_{\eta_2})$ with $\theta-\vT_{\eta_2}\in(-\theta_q, \theta_q)$.
\end{proof}

\subsection{The Asymptotics for $q>1$} 
We shall prove Theorem~\ref{thm:main1Star} by contradiction. To that end, assume that there are infinitely many $k_n$'s such that \eqref{eq:ITEP} admits a solution $(u_{k_n},v_{k_n})$ with $v_{k_n}=\He[k_n,\phi]$ for a given nontrivial $C^1$ function $\phi$. Then the principal term of $\Int^{(N)}(k)$, given by \eqref{eq:asympGen_diffeta} , will tend to $0$ as $k=k_n\to\infty$, for all $\eta\in\Ss^1$ and all $N\in\mathbb{N}$.

We observe from \eqref{eq:PsiStationaryStar}, \eqref{eq:HessianStar} and \eqref{eq:cossintStarBig} that
\begin{equation}\label{eq:sgnPsiStarBig}
	\sgn \Psi_{\eta}^{j,l}=\sgn \psi_{\eta}^{j,l}=(-1)^{l+j}
	\AND
	\sgn \Tr  D^2\psi_{\eta}^{1,l}=-(-1)^{l-1}
\end{equation}
where
\begin{equation}
	\Brak{\Psi_{\eta}^{j,l},\psi_{\eta}^{j,l}, D^2\psi_{\eta}^{j,l}} = \brak{\Psi_{\eta},\psi_{\eta}, D^2\psi_{\eta}}(\Cp_{j,l}\eta\,,\,\Cp_{j,l}\eta+\delta_{l,2}\pi),\qquad j,l=1,2.
\end{equation}
Applying \eqref{eq:cost2StarBig} to \eqref{eq:HessianStar} yields
\begin{equation*}
	(-1)^{j-1}\det(D^2\psi_{\eta})(\theta,\vT_{\xi})
	=\pare{\sqrt{q+(q-1)\rn'^2}+(-1)^{j}(\vec{\theta}\cdot\eta_l\sqrt{q}+1)\rn''}|_\theta \rho^2(\theta),
\end{equation*}
with $\theta=\Cp_{j,l}\eta$ and $\xi=(-1)^{l-1}\vec{\theta}$.
Combining this with \eqref{eq:StarBigCond1}, \eqref{eq:cossintStarBig} and \eqref{eq:sgnPsiStarBig} we derive that
\begin{equation}\label{eq:sgnDpsiStarBig}
	\sgn \det D^2\psi_{\eta}^{j,l}=(-1)^{j-1}
	\AND \sgn D^2\psi_{\eta}^{j,l}=(-1)^l-(-1)^{l+j}.
\end{equation}
Applying Lemma~\ref{lem:4criticStarBig} to the principal term of \eqref{eq:asympGen_diffeta} we now obtain
\begin{equation}\label{eq:asympStarBig}
	\sum_{j,l=1}^{2}
	\text{\small$\pare{f_{\eta}^{j,l}}^N
		\frac{\phi(\Cp_{j,l}\eta+\delta_{l,2}\pi)\Psi_{\eta}^{j,l}}{\Abs{\det D^2\psi_{\eta}^{j,l}}^{1/2}}$}\,
	e^{\im\pi\frac{(-1)^l (1-(-1)^j)}{4}+\im k\psi_{\eta}^{j,l}}
	\to0,
	\qquad
	\mbox{for all $\eta\in\Ss^1$ and $N\in\mathbb{N}$,}
\end{equation}
as $k=k_n\to\infty$. Here
\begin{equation}\label{eq:fjlStar}
	f_{\eta}^{j,l}=\rho(\Cp_{j,l}\eta) \sin(\Cp_{j,l}\eta-\vT_{\eta}),\qquad j,l=1,2.
\end{equation}
For each $\eta\in\Ss^1$, denote
\begin{equation}\label{eq:Lambda}
	\mbox{$\Lambda=\Lambda_{\eta}=\{(j,l); \phi(\Cp_{j,l}\eta+\delta_{l,2}\pi)\neq 0, \,j,l=1,2\}$. }
\end{equation}

Then
\begin{equation}\label{eq:asympN0StarBig}
	\sum_{(j,l)\in\Lambda}\pare{f_{\eta}^{j,l}}^N
	\frac{\phi(\Cp_{j,l}\eta+\delta_{l,2}\pi)\Psi_{\eta}^{j,l}}{\Abs{\det D^2\psi_{\eta}^{j,l}}^{1/2}}\,e^{\im\pi\frac{(-1)^l (1-(-1)^j)}{4}+\im k\psi_{\eta}^{j,l}}
	\to0,
	\qquad
	\mbox{for all $\eta\in\Ss^1$ and $N\in\mathbb{N}$,}
\end{equation}
as $k=k_n\to\infty$. We have the following result depending on $\#\Lambda$, the number of elements in the set $\Lambda$.
\begin{lem}\label{lem:asympStarBig}
	Under the same assumptions and notations as in Lemma~\ref{lem:4criticStarBig}, suppose that \eqref{eq:asympStarBig} holds true.
	Then $\#\Lambda\in\{0,2,3,4\}$ with the following additional properties satisfied.

	If $\#\Lambda\in\{2,3\}$, then $f_{\eta}^{j,l}$ take the same value  for all $(j,l)\in\Lambda$.

	If $\#\Lambda=4$, then either $f_{\eta}^{j,l}$ take the same value for all $(j,l)\in\Lambda$, or $\Lambda=\Lambda_1\cup\Lambda_2$ where for each $p=1,2$, $\#\Lambda_p=2$ and $f_{\eta}^{j,l}$ take the same value for all $(j,l)\in\Lambda_p$.  In the latter case, \eqref{eq:asympN0StarBig} is satisfied for each $\Lambda_p$ (with $\Lambda$ replaced by $\Lambda_p$), $p=1,2$.
\end{lem}
\begin{proof}
	Recall from Lemma \ref{lem:Psi=0} that $\Psi_{\eta}^{j,l}\neq 0$, $j,l=1,2$. Consider now the case $\#\Lambda=4$, and let $t_\iota, a_\iota$ denote
\begin{equation*}
		t_\iota(k)=\text{\small$\frac{\phi(\Cp_{j_\iota,l_\iota}\eta+\delta_{l_\iota,2}\pi)\Psi_{\eta}^{j_\iota,l_\iota}}{\Abs{\det D^2\psi_{\eta}^{j_\iota,l_\iota}}^{1/2}}$}\,e^{\im\pi\frac{(-1)^{l_\iota} (1-(-1)^{j_\iota})}{4}+\im k\psi_{\eta}^{j_\iota,l_\iota}}\neq 0
		\AnD a_\iota=f_{\eta}^{j_\iota,l_\iota},
		\quad\mbox{for $\iota=1,\ldots,4$},
	\end{equation*}
with $\cup_{\iota=1}^4\{(j_\iota,l_\iota)\}=\Lambda$.	After extraction of a subsequence we may assume that
$$
t_\iota(k) \to z_\iota \ne 0 \mbox{ as } k \to \infty~.
$$
Applying \eqref{eq:asympStarBig} with $N=0,1,2,3$ we obtain that
	\begin{equation*}
		\begin{bmatrix}
			1&1&1&1\\a_1&a_2&a_3&a_4\\a_1^2&a_2^2&a_3^2&a_4^2\\a_1^3&a_2^3&a_3^3&a_4^3
		\end{bmatrix}
		\begin{bmatrix}
			z_1\\z_2\\z_3\\z_4
		\end{bmatrix}
		= 0~~.
	\end{equation*}
	By direct algebraic operations, on the 4-by-4 Vandermonde coefficient matrix, we obtain that either $a_1=a_2=a_3=a_4$ and $z_1+z_2+z_3+z_4=0$, or, up to a swap of notations, $a_1=a_2\neq a_3=a_4$ and $z_1+z_2=0$ and $z_3+z_4=0$. The fact that \eqref{eq:asympN0StarBig} holds for each $\Lambda_p$ follows immediately from this latter statement.
	By similar arguments we can verify the cases when $\#\Lambda=2$ or $3$, as well as show that $\#\Lambda\neq 1$.
\end{proof}
The following results are consequences of Lemmas~\ref{lem:crimoreStarBig}~and~\ref{lem:asympStarBig}.
\begin{cor}\label{cor:phiStarBig}
	Under the same assumptions and notations as in Lemma~\ref{lem:asympStarBig}, let $\eta\in\Ss^1$ satisfy $\rho'(\vT_\eta)\rho'(\vT_\eta+\pi)\neq 0$.
	Then $\#\Lambda\in\{0,2,4\}$.
	Moreover, if $\#\Lambda=4$ then
	\begin{equation*}
		f_{\eta}^{j,1}=f_{\eta}^{j,2}
		\AND \frac{\abs{\phi(\Cp_{j,1}\eta)}\Abs{\Psi_{\eta}^{j,{1}}}}{\Abs{\det D^2\psi_{\eta}^{j,1}}^{1/2}}
		=\frac{\abs{\phi(\Cp_{j,2}\eta+\pi)}\Abs{\Psi_{\eta}^{j,2}}}{\Abs{\det D^2\psi_{\eta}^{j,2}}^{1/2}},\qquad j=1,2.
	\end{equation*}
	If $\#\Lambda=2$ then $\Lambda=\{(j_1,l_1),(j_2,l_2)\}\neq \{(j,1),(j+1,2)\}$, $j=1,2$, and
	\begin{equation*}
		f_{\eta}^{j_1,l_1}=f_{\eta}^{j_2,l_2}
		\AND \frac{\abs{\phi(\Cp_{j_1,l_1}\eta+\delta_{l_1,2}\pi)}\Abs{\Psi_{\eta}^{j_1,l_1}}}{\Abs{\det D^2\psi_{\eta}^{j_1,l_1}}^{1/2}}
		=\frac{\abs{\phi(\Cp_{j_2,l_2}\eta+\delta_{l_2,2}\pi)}\Abs{\Psi_{\eta}^{j_2,l_2}}}{\Abs{\det D^2\psi_{\eta}^{j_2,l_2}}^{1/2}}.
	\end{equation*}
\end{cor}
\begin{proof}
	Recall from Lemma~\ref{lem:crimoreStarBig} that
	\begin{equation}\label{eq:feta}
			|f_\eta^{2,1}|<|f_\eta^{1,2}|\AND |f_\eta^{2,2}|<|f_\eta^{1,1}|.
	\end{equation}
	Then we immediately infer from Lemma~\ref{lem:asympStarBig} that $\#\Lambda\neq 3$, which require three of the $f^{j,l}_\eta$'s coincide.
	If $\#\Lambda=4$, by Lemma~\ref{lem:asympStarBig}, we must have $f^{j_1,l_1}_\eta=f^{j_2,l_2}_\eta$ and $f^{j_3,l_3}_\eta=f^{j_4,l_4}_\eta$ for some $\{(j_\iota,l_\iota);\iota=1,2,3,4\}=\{(1,1),(1,2),(2,1),(2,2)\}$. In view of \eqref{eq:feta} we then must have $f_{\eta}^{1,1}=f_{\eta}^{1,2}$ and $f_{\eta}^{2,2}=f_{\eta}^{2,1}$. Moreover, notice that
	\begin{equation*}
		\text{\small$\frac{\phi(\Cp_{j,l}\eta+\delta_{l,2}\pi)\Psi_{\eta}^{j,l}}{\Abs{\det D^2\psi_{\eta}^{j,l}}^{1/2}}$}
		\mbox{ is independent of $k$,}
		\AnD\Abs{e^{\im\pi\frac{(-1)^l (1-(-1)^j)}{4}+\im k\psi_{\eta}^{j,l}}}=1,
	\end{equation*}
	Then we deduce from \eqref{eq:asympN0StarBig} for each $\Lambda_j$ that
	\begin{equation*}
		\frac{\abs{\phi(\Cp_{j,1}\eta)}\Abs{\Psi_{\eta}^{j,{1}}}}{\Abs{\det D^2\psi_{\eta}^{j,1}}^{1/2}}
		=\frac{\abs{\phi(\Cp_{j,2}\eta+\pi)}\Abs{\Psi_{\eta}^{j,2}}}{\Abs{\det D^2\psi_{\eta}^{j,2}}^{1/2}},\qquad j=1,2.
	\end{equation*}
	The case when $\#\Lambda=2$ can be shown similarly.
\end{proof}
\begin{cor}\label{cor:phi2StarBig}
	Under the same assumptions and notations as in Lemma~\ref{lem:asympStarBig}, suppose $\eta\in\Ss^1$ satisfies $\rho'(\vT_{\eta}+\pi)=0\neq \rho'(\vT_{\eta})$. Then 	$\phi(\Cp_{1,1}\eta)=\phi(\Cp_{2,2}\eta+\pi)=0$. Moreover, $\phi(\vT_\eta)$ and $\phi(\vT_\eta+\pi)$ are either both zero or both nonzero.
\end{cor}
\begin{proof}
	Since $\rho'(\vT_\eta+\pi)=0\neq \rho'(\vT_\eta)$, by \eqref{eq:critDiffrStarBig} we have $\Cp_{2,1}\eta=\Cp_{1,2}\eta=\vT_\eta+\pi$, and so $f_\eta^{1,2}=f_\eta^{2,1}=0$. By Lemma~\ref{lem:crimoreStarBig},  $0<|f_\eta^{2,2}|<|f_\eta^{1,1}|$. Consequently we obtain from Lemma~\ref{lem:asympStarBig} that $\#\Lambda\in\{0,2\}$. Moreover, if $\#\Lambda=0$ then $\phi(\Cp_{1,1}\eta)=\phi(\Cp_{2,2}\eta+\pi)=\phi(\vT_\eta)=\phi(\vT_\eta+\pi)=0$. If $\#\Lambda=2$ then $\phi(\Cp_{1,1}\eta)=\phi(\Cp_{2,2}\eta+\pi)=0$, $\phi(\vT_\eta)\phi(\vT_\eta+\pi)\neq 0$.
\end{proof}

\subsection{Proofs of Theorems~\ref{thm:Star1},
	~\ref{thm:main1Star}~and~\ref{thm:StarAnly}}
\label{sec:prfStarBig}
We start these proofs in the same way as that of Theorem~\ref{thm:ellipse}, by assuming the opposite, namely that there are infinitely many wave numbers $k_n$, $n\in\NN$, for which \eqref{eq:ITEP} ($\lambda=k_n^2$) admits a solution $(u_{k_n},v_{k_n})$ with $v_{k_n}=\He[k_n,\phi]$.
Then $k_n\to \infty$ and consequently, the asymptotics \eqref{eq:asympStarBig} as well as Corollaries~\ref{cor:phiStarBig}~and~\ref{cor:phi2StarBig} are valid. Moreover, all the results established in Section~\ref{sec:statStarBig} hold true. In each case we then show that this leads to a contradiction to either $\phi$ being non-trivial or the conditions on $\rho$. In order to obtain such a contradiction, some of the ideas are similar to those in the proof of Theorem~\ref{thm:ellipse}: we establish from the asymptotics \eqref{eq:asympN0StarBig} an identity similar to \eqref{eq:proofIdEll} and show that $\Gp(0)<1$ which, after we prove $\phi(0)=0$, yields $\phi=0$ near $0^+$. However, in order to show that $\phi\equiv 0$ globally, we need to make full use of the counterpart of $f_{\eta,1}=f_{\eta,2}$ in \eqref{eq:prooff}, which is Corollary~\ref{cor:phiStarBig} for the star-shaped case. The mapping properties of $\Cp_{j,l}$, $j,l=1,2$, established in Section~\ref{sec:statStarBig}, play an important role in these proofs.

\subsubsection{Proof of Theorem~\ref{thm:Star1}}
	Given the assumptions of Theorem \ref{thm:Star1} there exists some open interval where $\rho'$ has one sign, say $\rho'>0$. We may assume, up to a rotational change of coordinates, that
	\begin{equation}\label{eq:prfrhoStarBig}
		\mbox{$\rho'>0$ \quad in $(0,\tau)$ \AND $\rho'(0)=\rho'(\tau)=0$}.
	\end{equation}
	Then, from \eqref{eq:critDiffrStarBig} and the assumptions of Theorem \ref{thm:Star1},
	$$\Cp_{1,1}(0,\tau)=\Cp_{2,2}(0,\tau)=(0,\tau),\AND
		\mbox{$\rho'=0$ \quad in $(\pi,\pi+\tau)$}.$$
	Hence by Corollary~\ref{cor:phi2StarBig} it follows that
	$$\mbox{$\phi=0$\qquad in $\Cp_{1,1}(0,\tau)\cup\pare{\Cp_{2,2}(0,\tau)+\pi}=(0,\tau)\cup(\pi,\tau+\pi)$}.$$
	Applying this argument to all connected components of the set $\{ \rho'\ne 0\}$ we obtain that
	$$\mbox{$\phi=0$\qquad in $\overline{\N\cup\pare{\N+\pi}}=[0,2\pi)$},$$
	which contradicts the assumption that $\phi$ is nontrivial. The proof is complete.

\subsubsection{Proof of Theorem~\ref{thm:StarAnly}}
	 Similar to before, we assume without loss of generality that \eqref{eq:prfrhoStarBig} holds for some $\tau>0$.
	As a consequence, $\rho''(0)\ge 0\ge\rho''(\tau)$. Notice that $$\Cp_{1,1}(0,\tau)=\Cp_{2,2}(0,\tau)=(0,\tau).$$ Then by Lemmas~\ref{lem:critMapStarBig}~and~\ref{lem:crimoreStarBig} we obtain for all $t\in(0,\tau)$ that
	\begin{equation*}
		0<\sin(\Cp_{2,2}t-t)<\sin(\Cp_{1,1}t-t),
	\end{equation*}
	\begin{equation}\label{eq:prfT11StarBig}
		0<\Cp_{2,2}t-t<\Cp_{1,1}t-t<\pi-\theta_q
		\AND 0<\Cp_{1,1}^{-1}t<\Cp_{2,2}^{-1}t<t<\tau.
	\end{equation}
	Furthermore, since $\phi$ is a nontrivial $2\pi$-periodic real-analytic function, we derive from Corollaries~\ref{cor:phiStarBig}~and~\ref{cor:phi2StarBig} that, for all $t\in(0,\tau)\cup(\pi,\tau+\pi)$ except for a possible finite set of points, $$\rho'(t+\pi)\phi(\Cp_{1,1}t)\phi(\Cp_{2,1}t)\phi(\Cp_{1,2}t+\pi)\phi(\Cp_{2,2}t+\pi)\neq 0,$$ and
	\begin{equation}\label{eq:prfrho1112StarBig}
		\rho(\Cp_{1,1}t) \sin(\Cp_{1,1}t-t)=\rho(\Cp_{1,2}t) \sin(\Cp_{1,2}t-t)
		\AnD
		|\phi(\Cp_{1,1}t)|^2=\Gp_{11,12}(t)|\phi(\Cp_{1,2}t-\pi)|^2,
	\end{equation}
	\begin{equation}\label{eq:prfrho2221StarBig}
		\rho(\Cp_{2,2}t) \sin(\Cp_{2,2}t-t)=\rho(\Cp_{2,1}t) \sin(\Cp_{2,1}t-t)
		\AnD
		|\phi(\Cp_{2,2}t+\pi)|^2=\Gp_{22,21}(t)|\phi(\Cp_{2,1}t)|^2,
	\end{equation}
where
	\begin{equation}\label{eq:prfGjlStar}
		\Gp_{j_1l_1,j_2l_2}(t)=\frac{\Abs{\det D^2\psi_t^{j_1,l_1}}}{\Abs{\det D^2\psi_t^{j_2,l_2}}}\frac{\Abs{\Psi_t^{j_2,l_2}}^2}{\Abs{\Psi_t^{j_1,l_1}}^2},\qquad j_1,j_2,l_1,l_2=1,2.
	\end{equation}
	By continuity we then observe that \eqref{eq:prfrho1112StarBig} and \eqref{eq:prfrho2221StarBig} are satisfied for all $t\in[0,\tau]$.
	As a consequence, $\Cp_{1,2}0=\Cp_{2,1}0=\pi$, $\Cp_{1,2}\tau=\Cp_{2,1}\tau=\tau+\pi$ and $\sin(\Cp_{1,2}t-t)>0$ for all $t\in(0,\tau)$. Hence by Lemma~\ref{lem:crimoreStarBig} we have that $\rho'<0$ in $(\pi,\tau+\pi)$ and thus $\rho''(\pi)\le0\le\rho''(\tau+\pi)$. Moreover, for all $t\in(0,\tau)$ we have
	\begin{equation}\label{eq:prfT12StarBig}
		t+\theta_q<\Cp_{1,2}t<t+\pi
	\AND
		0<\sin(\Cp_{2,1}t-t)<\sin(\Cp_{1,2}t-t).
	\end{equation}
	In particular, we have shown that \eqref{eq:prfrhoStarBig} implies
	\begin{equation}\label{eq:prfrho2StarBig}
		\mbox{$\rho'<0$ \quad in $(\pi,\tau+\pi)$ \AND $\rho'(\pi)=\rho'(\tau+\pi)=0$},
	\end{equation}
	provided that there are infinitely many wave numbers $k_n$, $n\in\NN$, with which \eqref{eq:ITEP} admits a solution $(u_{k_n},v_{k_n})$ where $v_{k_n}=\He[k_n,\phi]$.
	In fact, by analogous arguments we can show that \eqref{eq:prfrhoStarBig} and \eqref{eq:prfrho2StarBig} are equivalent.
	Therefore, we have completed the proof of Theorem~\ref{thm:StarAnly} in Case (i) (recall:$q>1$).
	Direct calculation yields
	\begin{equation*}
		0<\Gp_{11,12}(0)=\frac{\sqrt{q}-\pare{\sqrt{q}+1}\rn''(0)}{\sqrt{q}-\pare{\sqrt{q}+1}\rn''(\pi)}.
	\end{equation*}
	In Case (ii) of Theorem~\ref{thm:StarAnly}, since we already have $\rho'(0)=\rho'(\pi)=0$, then $\rho''(0)$ and $\rho''(\pi)$ cannot both be zero. Therefore $\rho''(0)>\rho''(\pi)$ and hence $\Gp_{11,12}(0)<1$. Then, in view of \eqref{eq:prfrho1112StarBig}, we must have $\phi(0)=0$. In addition, by continuity we also have $\Gp_{11,12}<1$ in $(0,t_1)$ for some $t_1\in(0,\tau)$. Since $\phi$ is real analytic and nontrivial there exists $t_2\in(0,t_1)$ such that
	\begin{equation}\label{eq:prfNormStarBig}
		|\phi(t_2)|=\|\phi\|_{C^0[0,t_1]}>0\AND |\phi(t)|<|\phi(t_2)|\quad\mbox{for all $t\in(0,t_2)$}.
	\end{equation}
	Then we can apply \eqref{eq:prfrho1112StarBig} for $t=\Cp_{1,1}^{-1}t_2\in(0,t_2)$ and for $t=\Cp_{1,1}^{-1}\Cp_0t_2\in(0,t_2)$ where
	$\Cp_0=\Cp_{1,2}\Cp_{1,1}^{-1}-\pi$\footnote{Hereafter, as an operator, $\pm\pi$ is understood as $\pm\pi:t\mapsto t\pm\pi$.},
	which yields
	\begin{equation*}
		|\phi(t_2)|^2=\Gp(\Cp_{1,1}^{-1}t_2)|\phi(\Cp_{0}t_2)|^2=\Gp(\Cp_{1,1}^{-1}t_2)\Gp(\Cp_{1,1}^{-1}\Cp_0t_2)|\phi(\Cp_0^2t_2)|^2<|\phi(\Cp_0^2t_2)|^2.
	\end{equation*}

	However, this contradicts \eqref{eq:prfNormStarBig} because by \eqref{eq:prfT11StarBig} and \eqref{eq:prfT12StarBig} we have  $0<\Cp_{0}t=\Cp_{1,2}\Cp_{1,1}^{-1}t-\pi<\Cp_{1,1}^{-1}t<t$ for all $t\in(0,\tau)$ and hence $0<\Cp_{0}^2t_2<t_2$. The proof is complete.

\subsubsection{Proof of Theorem~\ref{thm:main1Star}}
For either Case (i) or Case (ii) in Theorem~\ref{thm:main1Star} we have that $\rho'(t)=0$ implies $\rho'(t+\pi)=0$. Let $\omega$ be a connected component of the set $\{ \theta :~\rho'(\theta) >0 \}$. We may without loss of generality assume that $\omega=(0,\tau)$ with $\tau \in (0,\pi]$. Then
\begin{equation}\label{eq:rho'proof}
	\mbox{$\rho'>0$\quad in $(0,\tau)$ \AND $\rho'(0)=\rho'(\tau)=0$\quad and }
\end{equation}
\begin{equation*}
	\mbox{$\Cp_{1,1}(0,\tau)=\Cp_{2,2}(0,\tau)=(0,\tau)=\Cp_{1,2}(\pi,\tau+\pi)=\Cp_{2,1}(\pi,\tau+\pi)
		$}.
\end{equation*}
In addition, we also have that $\rho'(\pi)=\rho'(\tau+\pi)=0$ and $\rho'(t+\pi)\neq 0$ for $t\in(0,\tau)$. Hence
\begin{equation*}
	\mbox{$\Cp_{1,2}(0,\tau)=\Cp_{2,1}(0,\tau)=(\pi,\tau+\pi)=\Cp_{1,1}(\pi,\tau+\pi)=\Cp_{2,2}(\pi,\tau+\pi)
		$}.
\end{equation*}
Moreover, by Lemmas~\ref{lem:critMapStarBig}~and~\ref{lem:crimoreStarBig} we obtain for all $t\in(0,\tau)$ that
\begin{equation}\label{eq:prfsinT}
0<\sin(\Cp_{2,2}t-t)<\sin(\Cp_{1,1}t-t)\AND t<\Cp_{2,2}t<\Cp_{1,1}t<t+\pi-\theta_q.
\end{equation}
We first prove that
\begin{equation}\label{eq:prfFinal}
	\mbox{$\phi=0$\qquad in\qquad $(0,\tau)$.}
\end{equation}

\noindent
Assume otherwise, namely
\begin{equation*}
\phi(\tau_{0})\neq 0\qquad \mbox{for some $\tau_{0}\in(0,\tau)$}.
\end{equation*}
Then we can find $t_1,t_2$ such that
\begin{equation}\label{eq:phiprf}
	\mbox{$\phi(t)\neq 0$ \qquad for $t\in(t_2,t_1)$,\quad with $0\le t_2<\tau_0<t_1\le\tau$,}
\end{equation}
and that for each $\iota=1,2$, either $\phi(t_\iota)=0$ or $t_\iota\in\{0,\tau\}$.
Notice from \eqref{eq:phiprf} that $\phi\circ\Cp_{1,1}$ is nowhere zero on $\Cp_{1,1}^{-1}(t_2,t_1)\subseteq(0,\tau)$. Then $\#\Lambda_t\ge 1$ for $t\in \Cp_{1,1}^{-1}(t_2,t_1)$. Hence by Corollary~\ref{cor:phiStarBig} we must have $\#\Lambda_t=2$ or $4$. Moreover, if $\#\Lambda_t=4$, then $\phi(\Cp_{j,1}t)\phi(\Cp_{j,2}t+\pi)\neq 0$ and $f^{j,1}_t=f^{j,2}_t$, for both $j=1,2$, where $f^{j,l}_t$ is defined in \eqref{eq:fjlStar}.
Otherwise if $\#\Lambda_t=2$, then $\phi(\Cp_{2,2}t+\pi)=0$, and either, $\phi(\Cp_{2,1}t)= 0\neq \phi(\Cp_{1,2}t-\pi)$ and $f^{1,1}_t=f^{1,2}_t$, or, $\phi(\Cp_{1,2}t-\pi)= 0\neq \phi(\Cp_{2,1}t)$ and $f^{1,1}_t=f^{2,1}_t$.
As a conclusion, whether $\#\Lambda_t=2$ or $4$, one of the following must hold true for each $t\in \Cp_{1,1}^{-1}(t_2,t_1)$:
\begin{equation}\label{eq:prfrho1112StarBig2}
	\rho(\Cp_{1,1}t) \sin(\Cp_{1,1}t-t)=\rho(\Cp_{1,2}t) \sin(\Cp_{1,2}t-t)
	\AnD
	|\phi(\Cp_{1,1}t)|^2=\Gp_{11,12}(t)|\phi(\Cp_{1,2}t-\pi)|^2,
\end{equation}
or
\begin{equation}\label{eq:prfrho1121StarBig}
	\rho(\Cp_{1,1}t) \sin(\Cp_{1,1}t-t)=\rho(\Cp_{2,1}t) \sin(\Cp_{2,1}t-t)
	\AnD
	|\phi(\Cp_{1,1}t)|^2=\Gp_{11,21}(t)|\phi(\Cp_{2,1}t)|^2.
\end{equation}
Combining with \eqref{eq:prfsinT} we infer that $\sin(\Cp_{1,2}t-t)$ or $\sin(\Cp_{2,1}t-t)$ must be positive,
and hence by Lemma~\ref{lem:crimoreStarBig} $\rho'(\Cp_{1,2}t)<0$ or equivalently $\rho'(\Cp_{2,1}t)<0$, for $t\in \Cp_{1,1}^{-1}(t_2,t_1)\subset(0,\tau)$.
Since $\rho'$ is nowhere zero in $(\pi,\pi+\tau)$, we must have $\rho'<0$ in $(\pi,\tau+\pi)$. For Case (i), this represents a contradiction to $\sgn\rho'(t)=\sgn\rho'(t+\pi)$, and so in this case \eqref{eq:prfFinal} is proven. We proceed to establish \eqref{eq:prfFinal} in Case (ii). Since $\rho'<0$ in $(\pi,\tau+\pi)$ a combination of Lemma~\ref{lem:crimoreStarBig} and \eqref{eq:prfsinT} gives that for all $t \in (0,\tau)$
\begin{equation*}
0<\sin(\Cp_{2,1}t-t)<\sin(\Cp_{1,2}t-t) \quad \hbox{ and } \quad t+\theta_q<\Cp_{1,2}t<\Cp_{2,1}t<t+\pi
\end{equation*}
and for all $t \in (0,\tau)$, $j=1,2$
\begin{equation}\label{eq:prfcp2122contr}
		\Cp_{1,2}t-\pi<\Cp_{2,1}t-\pi<t<\Cp_{2,2}t<\Cp_{1,1}t,\quad
		0<\Cp_{1,2}\Cp_{j,j}^{-1}t-\pi<t,\quad 0<\Cp_{2,1}\Cp_{j,j}^{-1}t-\pi<t.
\end{equation}
The following two auxiliary results will be needed.
\begin{lem}\label{lem:prfphi1112}
	Suppose that $\phi$ is nowhere zero in $(t_2,t_1)\subset(0,\tau)$. Then $\phi$ is either nowhere zero in $\Cp_{1,2}\Cp_{1,1}^{-1}(t_2,t_1)-\pi$, or identically zero in $\Cp_{1,2}\Cp_{1,1}^{-1}(t_2,t_1)-\pi$ and nowhere zero in $\Cp_{2,1}\Cp_{1,1}^{-1}(t_2,t_1)$.

	Similarly if $\phi$ is nowhere zero in $(t'_2+\pi,t'_1+\pi)\subset(\pi,\tau+\pi)$, then $\phi$ is either nowhere zero in $\Cp_{2,1}\Cp_{2,2}^{-1}(t'_2,t'_1)$, or identically zero in $\Cp_{2,1}\Cp_{2,2}^{-1}(t'_2,t'_1)$ and nowhere zero in $\Cp_{1,2}\Cp_{2,2}^{-1}(t_2',t_1')-\pi$.
\end{lem}
\begin{proof}
We first consider the case when $\phi$ is nowhere zero in $(t_2,t_1)$. Assume that $\phi$ is neither nowhere zero or identically zero in $\Cp_{1,2}\Cp_{1,1}^{-1}(t_2,t_1)-\pi$. Then there exists $t_{12}\in (t_2,t_1)$ and $\varepsilon_1\neq0$ such that
	\begin{equation}\label{eq:phin0prf}
		\mbox{$\phi(\Cp_{1,2}\Cp_{1,1}^{-1}t_{12}-\pi)=0$ \AND $\phi(\Cp_{1,2}t-\pi)\neq0$ \quad for all $t\in\Cp_{1,1}^{-1}(t_{12},t_{12}+\varepsilon_1)$.}
	\end{equation}
	Here, we allow $\varepsilon_1$ being either positive or negative, and in the latter case the interval $(t_{12},t_{12}+\varepsilon_1)$ is understood as $(t_{12}+\varepsilon_1,t_{12})$.
	Recall that one of the relations \eqref{eq:prfrho1112StarBig2} and \eqref{eq:prfrho1121StarBig} must be valid at $t=\Cp_{1,1}^{-1}t_{12}$, while the case \eqref{eq:prfrho1112StarBig2} contradicts the assumption $\phi(\Cp_{1,2}\Cp_{1,1}^{-1}t_{12}-\pi)=0$. Hence \eqref{eq:prfrho1121StarBig} must hold
	true for $t=\Cp_{1,1}^{-1}t_{12}$. As a consequence, we have $\phi(\Cp_{2,1}\Cp_{1,1}^{-1}t_{12})\neq 0$ and thus
	\begin{equation*}
		\mbox{$\phi(\Cp_{2,1}t)\neq 0$ \qquad for all $t\in\Cp_{1,1}^{-1}(t_{12}-\varepsilon_2,t_{12}+\varepsilon_2)$},
	\end{equation*}
	with some constant $\varepsilon_2\neq 0$. Therefore, there exists a constant $\varepsilon\neq 0$ such that
	\begin{equation*}
		\phi(\Cp_{1,1}t)\phi(\Cp_{1,2}t-\pi)\phi(\Cp_{2,1}t)\neq 0\qquad\mbox{for all $t\in\Cp_{1,1}^{-1}(t_{12},t_{12}+\varepsilon)$}.
	\end{equation*}
	Applying Corollary~\ref{cor:phiStarBig} we then derive that $\#\Lambda_t=4$ and \eqref{eq:prfrho1112StarBig2} (as well as \eqref{eq:prfrho2221StarBig}) must hold for all $t\in\Cp_{1,1}^{-1}(t_{12},t_{12}+\varepsilon)$. Thus by continuity \eqref{eq:prfrho1112StarBig2} is valid for $t=\Cp_{1,1}^{-1}t_{12}$, which implies that $\phi(\Cp_{1,2}\Cp_{1,1}^{-1}t_{12}-\pi)\neq 0$, contradicting \eqref{eq:phin0prf}.

	If $\phi=0$ in $\Cp_{1,2}\Cp_{1,1}^{-1}(t_2,t_1)-\pi$, then $\phi(\Cp_{1,1}t)\neq0=\phi(\Cp_{1,2}t-\pi)$ for all $t\in\Cp_{1,1}^{-1}(t_2,t_1)$. Hence by Corollary~\ref{cor:phiStarBig} we must have $\phi(\Cp_{2,1}t)\neq0=\phi(\Cp_{2,2}t+\pi)$ for all $t\in\Cp_{1,1}^{-1}(t_2,t_1)$.

	The case when $\phi$ is nowhere zero in $(t'_2+\pi,t'_1+\pi)$ can be shown analogously by observing $\phi(\Cp_{2,2}t+\pi)\neq 0$ for $t\in\Cp_{2,2}^{-1}(t'_2,t'_1)$, assuming
	\begin{equation*}
		\mbox{$\phi(\Cp_{2,1}\Cp_{2,2}^{-1}t'_{12})=0$ \AND $\phi(\Cp_{2,1}t)\neq0$ \quad for all $t\in\Cp_{2,2}^{-1}(t'_{12},t'_{12}+\varepsilon_1)$,}
	\end{equation*}
	and applying Corollary~\ref{cor:phiStarBig} in $\Cp_{2,2}^{-1}(t'_{12},t'_{12}+\varepsilon_1)$.
\end{proof}

\noindent
Applying Lemma~\ref{lem:prfphi1112} we can prove the following.
\begin{cor}\label{cor:prfseqnon0}
	Suppose that $\phi$ is nowhere zero in $(t_2,t_1)\subset (0,\tau)$. Then there exists a sequence of intervals $(t_{n,2},t_{n,1})\subset(0,\tau)$, $n\in\mathbb{N}_+$, with $\{t_{n,1}\}_{n\in\mathbb{N}_+}$ monotonically decreasing to $0$, such that $\phi$ is either nowhere zero in all $(t_{n,2},t_{n,1})$, $n\in\mathbb{N}_+$, or nowhere zero in all $(t_{n,2}+\pi,t_{n,1}+\pi)$, $n\in\mathbb{N}_+$.
\end{cor}
\begin{proof}
	Denote
	\begin{equation*}
		\CP_{1,j}:=\Cp_{1,2}\Cp_{j,j}^{-1}-\pi\AnD \CP_{2,j}:=\Cp_{2,1}\Cp_{j,j}^{-1}-\pi,\qquad j=1,2.
	\end{equation*}
	Then \eqref{eq:prfcp2122contr} implies that $\CP_{j,l}$, $j,l=1,2$, are strictly decreasing mappings on $(0,\tau)$. We have from Lemma~\ref{lem:prfphi1112} that
	\begin{equation*}
		\mbox{$\phi$ is nowhere zero in $\CP_{1,1}(t_2,t_1)\subset(0,\tau)$ or in $\CP_{2,1}(t_2,t_1)+\pi\subset(\pi,\tau+\pi)$.}
	\end{equation*}
	In the former case, we set $(t_{1,2},t_{1,1}):=\CP_{1,1}(t_2,t_1)$, and further apply Lemma~\ref{lem:prfphi1112} on $(t_{1,2},t_{1,1})$. Then
	\begin{equation*}
		\mbox{$\phi$ is nowhere zero in $\CP_{1,1}(t_{1,2},t_{1,1})\subset(0,\tau)$ or in $\CP_{2,1}(t_{1,2},t_{1,1})+\pi\subset(\pi,\tau+\pi)$.}
	\end{equation*}
	For the latter case, let $(t_{1,2},t_{1,1}):=\CP_{2,1}(t_2,t_1)$. Applying Lemma~\ref{lem:prfphi1112} again on $(t_{1,2},t_{1,1})+\pi$ yields
	\begin{equation*}
		\mbox{$\phi$ is nowhere zero in $\CP_{2,2}(t_{1,2},t_{1,1})+\pi\subset(\pi,\tau+\pi)$ or in $\CP_{1,2}(t_{1,2},t_{1,1})\subset(0,\tau)$}.
	\end{equation*}
	By induction, for each $n\in\NN_+$, we can find $\CP^{(n)}:=\CP_{j,l}$ for some $j,l=1,2$, and $(t_{n+1,2},t_{n+1,1}):=\CP^{(n)}(t_{n,2},t_{n,1})$, such that $\phi$ is nowhere zero in $(t_{n+1,2},t_{n+1,1})$ or $(t_{n+1,2},t_{n+1,1})+\pi$.
	Therefore, we can find an infinite subsequence of $\{(t_{n,2},t_{n,1});n\in\NN_+\}$, say $\{(t'_{n,2},t'_{n,1});n\in\NN_+\}$, such that $\phi$ is nowhere zero either in all $(t'_{n,2},t'_{n,1})$, $n\in\mathbb{N}_+$, or in all $(t'_{n,2}+\pi,t'_{n,1}+\pi)$, $n\in\mathbb{N}_+$. Moreover $t_{n,1}\to 0$ as $n\to\infty$ since $\CP^{(n)}$ is a strictly decreasing mapping on $(0,\tau)$.
\end{proof}

Returning to the proof of Theorem~\ref{thm:main1Star}, we consider the case from Corollary~\ref{cor:prfseqnon0}, when $\phi$ is nowhere zero in $(t_{n,2},t_{n,1})\subset(0,\tau)$, $n\in\mathbb{N}_+$, with $t_{n,1}\to 0$ as $n\to\infty$. The case when $\phi$ is nowhere zero in $(t_{n,2}+\pi,t_{n,1}+\pi)$ can be shown analogously, but we opt to omit the details.

On each of the intervals $(t_{n,2},t_{n,1})$ we can apply Lemma~\ref{lem:prfphi1112} to obtain that one of the followings holds. Either $\phi$ is nowhere zero in $\Cp_{1,2}\Cp_{1,1}^{-1}(t_{n,2},t_{n,1})-\pi$ and hence by Corollary~\ref{cor:phiStarBig}, \eqref{eq:prfrho1112StarBig2} is satisfied in $\Cp_{1,1}^{-1}(t_{n,2},t_{n,1})$; in particular, \begin{equation}\label{eq:prfrhosin1112StarBig}
	\frac{d}{dt}\left[\rho(\Cp_{1,1}t) \sin(\Cp_{1,1}t-t)\right]=\frac{d}{dt}\left[\rho(\Cp_{1,2}t) \sin(\Cp_{1,2}t-t)\right],
\end{equation}
holds for $t\in\Cp_{1,1}^{-1}(t_{n,2},t_{n,1})$. Or $\phi=0$ in $\Cp_{1,2}\Cp_{1,1}^{-1}(t_{n,2},t_{n,1})-\pi$, and thus by Corollary~\ref{cor:phiStarBig} again, \eqref{eq:prfrho1121StarBig} is satisfied in $\Cp_{1,1}^{-1}(t_{n,2},t_{n,1})$; in particular,
\begin{equation}\label{eq:prfrhosin1121StarBig}
	\frac{d}{dt}\left[
	\rho(\Cp_{1,1}t) \sin(\Cp_{1,1}t-t)\right]=\frac{d}{dt}\left[\rho(\Cp_{2,1}t) \sin(\Cp_{2,1}t-t)\right],
\end{equation}
for $t\in\Cp_{1,1}^{-1}(t_{n,2},t_{n,1})$. Therefore, letting $n\to \infty$, by continuity we infer that
\begin{equation}\label{eq:prft=0}
	\mbox{either \eqref{eq:prfrhosin1112StarBig} or \eqref{eq:prfrhosin1121StarBig} is valid for $t=0$}.
\end{equation}

Notice that if $\phi$ is nowhere zero in an interval $(t_2,t_1)\subset(0,\tau)$, we also have $\phi(\Cp_{2,2}t)\neq 0$ for $t\in \Cp_{2,2}^{-1}(t_2,t_1)$. Recall from \eqref{eq:CppiStarBig} that
\begin{equation}\label{eq:prf1112pi}
	\Cp_{1,1}t=\Cp_{1,2}(t+\pi)\AND\Cp_{2,2}t=\Cp_{2,1}(t+\pi).
\end{equation}
Then $\phi(\Cp_{2,1}t)\neq 0$ for $t\in \Cp_{2,2}^{-1}(t_2,t_1)+\pi\subset(\pi,\pi+\tau)$. By Corollary~~\ref{cor:phiStarBig}, for each $t\in\Cp_{2,2}^{-1}(t_2,t_1)+\pi$, either $\#\Lambda_t=4$ with both \eqref{eq:prfrho1112StarBig} and \eqref{eq:prfrho2221StarBig} valid, or $\#\Lambda_t=2$ with either \eqref{eq:prfrho2221StarBig} or \eqref{eq:prfrho1121StarBig} satisfied. To fix the domain of the operators $\Cp_{i,j}$ as $(0,\tau)$, we make use of \eqref{eq:prf1112pi} again to rewrite \eqref{eq:prfrho2221StarBig} and  \eqref{eq:prfrho1121StarBig} as, respectively
\begin{equation}\label{eq:prfrho2122StarBig}
	\rho(\Cp_{2,2}t) \sin(\Cp_{2,2}t-t)=\rho(\Cp_{2,1}t) \sin(\Cp_{2,1}t-t)
	\AnD
	|\phi(\Cp_{2,2}t)|^2=\Gp_{22,21}(t)|\phi(\Cp_{2,1}t-\pi)|^2,
\end{equation}
and
\begin{equation}\label{eq:prfrho1222StarBig}
	\rho(\Cp_{2,2}t) \sin(\Cp_{2,2}t-t)=\rho(\Cp_{1,2}t) \sin(\Cp_{1,2}t-t)
	\AnD
	|\phi(\Cp_{2,2}t)|^2=\Gp_{22,12}(t)|\phi(\Cp_{1,2}t)|^2.
\end{equation}
Then applying similar arguments as in the proof of Lemma~\ref{lem:prfphi1112}, we can show that
\begin{equation}\label{eq:prfphi2122}
	\begin{split}
		&\mbox{$\phi$ being nowhere zero in   $(t_2,t_1) \subset (0,\tau)$ implies $\phi$ is either}
		\\& \mbox{nowhere zero in $\Cp_{2,1}\Cp_{2,2}^{-1}(t_2,t_1)-\pi$ or identically zero  }\\
		&\mbox{in $\Cp_{2,1}\Cp_{2,2}^{-1}(t_2,t_1)-\pi$ and nowhere zero in $\Cp_{1,2}\Cp_{2,2}^{-1}(t_2,t_1)$ ,}
	\end{split}
\end{equation}
and moreover,
\begin{equation}\label{eq:prfphi2case2122}
	\begin{split}
		&\mbox{if $\phi$ is nowhere zero in $\Cp_{2,1}\Cp_{2,2}^{-1}(t_2,t_1)-\pi$, then \eqref{eq:prfrho2122StarBig} holds }
		\\&\mbox{in $\Cp_{2,2}^{-1}[t_2,t_1]$; otherwise \eqref{eq:prfrho1222StarBig} holds in $\Cp_{2,2}^{-1}[t_2,t_1]$.}
	\end{split}
\end{equation}
Applying the statements \eqref{eq:prfphi2122} and \eqref{eq:prfphi2case2122}, we may, similarly to before (see Corollary~\ref{cor:prfseqnon0}) obtain a sequence of intervals $(t_{n,2},t_{n,1})\subset (0,\tau)$, with $t_{n,1}\to 0$ as $n \to \infty$ so that for all $n$ either
	\begin{equation}\label{eq:prfrhosin2221StarBig}
		\frac{d}{dt}\brak{\rho(\Cp_{2,2}t) \sin(\Cp_{2,2}t-t)}=\frac{d}{dt}\brak{\rho(\Cp_{2,1}t) \sin(\Cp_{2,1}t-t)},
	\end{equation}
	or
	\begin{equation}\label{eq:prfrhosin2212StarBig}
		\frac{d}{dt}\brak{\rho(\Cp_{2,2}t) \sin(\Cp_{2,2}t-t)}=\frac{d}{dt}\brak{\rho(\Cp_{1,2}t) \sin(\Cp_{1,2}t-t)},
	\end{equation}
	holds for all $t\in \Cp_{2,2}^{-1}(t_{n,2},t_{n,1})$. Therefore, passing to the limit $n \to \infty$, we get that
	\begin{equation}\label{eq:prft=02}
		\mbox{either \eqref{eq:prfrhosin2221StarBig} or \eqref{eq:prfrhosin2212StarBig} is valid for $t=0$}.
	\end{equation}
We conclude the proof of \eqref{eq:prfFinal}
by showing that \eqref{eq:prft=0} and \eqref{eq:prft=02} contradict each other.
Direct calculation yields
\begin{equation*}
	\begin{split}
		\frac{d}{dt}\brak{\rho(\Cp_{j,l}t)\sin(\Cp_{j,l}t-t)}
		&=\rho'(\Cp_{j,l}t)\sin(\Cp_{j,l}t-t)\Cp_{j,l}'t+\rho(\Cp_{j,l}t) \cos(\Cp_{j,l}t-t)\Pare{\Cp_{j,l}'t-1},
	\end{split}
\end{equation*}
and (see, \eqref{eq:dTstarbig})
\begin{equation*}
	\begin{split}
		\Cp_{j,l}'t-1=\frac{d\Cp_{j,l}t}{dt}-1
		&=\frac{Q_{j,l}(t)\rn''(\Cp_{j,l}t)}{R_{j,l}(t)-Q_{j,l}(t)\rn''(\Cp_{j,l}t)},
	\end{split}
\end{equation*}
where
\begin{equation*}
	R_{j,l}(t)=\sqrt{q+(q-1)\rn'^2(\Cp_{j,l}t)}
	\AND
	Q_{j,l}(t)=\frac{R_{j,l}(t)-(-1)^j}{1+\rn'^2(\Cp_{j,l}t)}.
\end{equation*}
Recall that $\rho''^2(0)+\rho''^2(\pi)\neq0$. Then either \eqref{eq:prfrhosin1112StarBig} or \eqref{eq:prfrhosin1121StarBig} at $t=0$ implies that $\rho''(0)\rho''(\pi)\neq 0$, more precisely that $\rho''(\pi)<0<\rho''(0)$. In addition,
\begin{equation}\label{prfFinalStarBig}
	\frac{\sqrt{q}}{\pare{\sqrt{q}+1}\rho''(0)}-\frac{1}{\rho(0)}
	=\frac{1}{\rho(\pi)}-\frac{\sqrt{q}}{\pare{\sqrt{q}-(-1)^{j}}\rho''(\pi)},
\end{equation}
where $j=1$ if \eqref{eq:prfrhosin1112StarBig} is satisfied at $t=0$ and $j=2$ if \eqref{eq:prfrhosin1121StarBig} holds for $t=0$. Similarly, we have
\begin{equation}\label{prfFinal2StarBig}
	\frac{\sqrt{q}}{\pare{\sqrt{q}-1}\rho''(0)}-\frac{1}{\rho(0)}
	=\frac{1}{\rho(\pi)}-\frac{\sqrt{q}}{\Pare{\sqrt{q}-(-1)^{\tilde{j}}}\rho''(\pi)},
\end{equation}
where $\tilde{j}=1$ if \eqref{eq:prfrhosin2212StarBig} is satisfied at $t=0$ and $\tilde{j}=2$ if \eqref{eq:prfrhosin2221StarBig} holds for $t=0$.
However,
we observe directly that \eqref{prfFinalStarBig} and \eqref{prfFinal2StarBig} can not be both true if $j=\tilde{j}$. In the case when $j=1$ and $\tilde{j}=2$ we derive from \eqref{prfFinalStarBig} and \eqref{prfFinal2StarBig} that $\rho''(0)=-\rho''(\pi)$ and $\rho(0)=-\rho(\pi)$, which violates $\rho>0$. Finally, if $j=2$ and $\tilde{j}=1$ we obtain that $\rho''(0)\rho''(\pi)>0$, which contradicts the assumption that $\rho''(\pi)<0<\rho''(0)$.

Up to now, we have proven that \eqref{eq:rho'proof} implies \eqref{eq:prfFinal}. That is, $\phi$ must be identically zero in any interval where $\rho'>0$. Applying analogous arguments we can also show if $\rho'<0$ in $(0,\tau)$ with $\rho'(0)=\rho'(\tau)=0$, then $\phi=0$ in $(0,\tau)$. In particular (if $\phi$ is not identically zero in $(0,\tau)$) we can deduce that $\rho'>0$ in $(\pi,\tau+\pi)$, and Lemma~\ref{lem:prfphi1112} as well as \eqref{eq:prfphi2122} and \eqref{eq:prfphi2case2122} are still valid. In place of \eqref{eq:prfcp2122contr} we have for all $t\in(0,\tau)$ that
\begin{equation}
		\Cp_{1,2}t-\pi>\Cp_{2,1}t-\pi>t>\Cp_{2,2}t>\Cp_{1,1}t,\quad
		\tau>\Cp_{1,2}\Cp_{j,j}^{-1}t-\pi>t,\quad \tau>\Cp_{2,1}\Cp_{j,j}^{-1}t-\pi>t;
\end{equation}
Corollary~\ref{cor:prfseqnon0} holds, except that $\{t_{n,2}\}_{n\in\mathbb{N}_+}$ is now a monotonically increasing sequence approaching $\tau$ (instead of $\{t_{n,1}\}_{n\in\mathbb{N}_+}$ decreasing to $0$), and \eqref{eq:prft=0} and \eqref{eq:prft=02} are valid at $t=\tau$. This leads to a contradiction in the same way as before. We leave the details to the reader. It is now established that $\phi$ is identically zero on the set $\{\theta:~\rho'(\theta)\neq 0\}$, and by continuity and the fact that $\rho'\neq 0$ almost everywhere it follows that $\phi$ vanishes identically on $[0,2\pi)$. The proof is complete.

\section*{Appendix}
\addcontentsline{toc}{section}{Appendix}
\begin{proof}[Proof of Theorem~\ref{thm:disk}]
	We can parameterize $\partial B_{R_0}(x_0)$ as a star domain with $$\rho(\theta)=|x_0|\cos(\theta-\vT_{x_0})+\sqrt{R_0^2-|x_0|^2\sin^2(\theta-\vT_{x_0})}.$$
	By direct calculations we derive that
	$$(\ln\rho)'(\theta)=\frac{-\sin(\theta-\vT_{x_0})}{\sqrt{R_0^2/|x_0|^2-\sin^2(\theta-\vT_{x_0})}}
	\AND (\ln\rho)''(\theta)=\frac{-R_0^2/|x_0|^2\cos(\theta-\vT_{x_0})}{\pare{R_0^2/|x_0|^2-\sin^2(\theta-\vT_{x_0})}^{3/2}}.$$
	Aside from the admissibility condition \eqref{eq:StarBigCond1}, it is evident that $\rho$ satisfies the conditions in Theorem~\ref{thm:main1Star}.

	We assume without loss of generality that $\theta_{x_0}=0$, and consequently $x_0=(R_0/s,0)$ with $s=R_0/|x_0|>1$. Then
	$$(\ln\rho)'(\theta)=\frac{-\sin\theta}{\sqrt{s^2-\sin^2\theta}}
	\AND (\ln\rho)''(\theta)=-s^2\frac{\cos\theta}{\pare{s^2-\sin^2\theta}^{3/2}}.$$
	It can be easily verified that
	\begin{equation*}
		-\min_\theta\, (\ln\rho)''(\theta)=\max_\theta\, (\ln\rho)''(\theta)=
		\begin{cases}
		(\ln\rho)''(\pi)=1/s,	&\qquad\mbox{if $s\ge\sqrt{3}$},
		\\
		(\ln\rho)''(\arccos-\sqrt{\dfrac{s^2-1}{2}})=\dfrac{2}{3\sqrt{3}}\dfrac{s^2}{s^2-1},	&\qquad\mbox{if $s<\sqrt{3}$}.
		\end{cases}
	\end{equation*}

When $\sqrt{q} \le 1/(\sqrt{3}-1)$ then $(1+\sqrt{q})/\sqrt{q}$ is greater than or equal to $\sqrt{3}$. Therefore $s= R_0/|x_0|>(1+\sqrt{q})/\sqrt{q}$  implies $s>\sqrt{3}$, and so
	$$(\ln \rho )''(\theta)\le \frac1{s}<\frac{\sqrt{q}}{1+\sqrt{q}} \mbox{ for all } \theta~,
	$$
i.e., \eqref{eq:StarBigCond1} is satisfied.

When $\sqrt{q} >	1/(\sqrt{3}-1)$ then $(1+\sqrt{q})/\sqrt{q}$ is strictly between $1$ and $\sqrt{3}$, and
$$
A_q:=1/\sqrt{1-2(1+\sqrt{q})/3\sqrt{3q}}>(1+\sqrt{q})/\sqrt{q}~.
$$
Therefore $s=R_0/|x_0|> A_q$ implies $s>(1+\sqrt{q})/\sqrt{q}$ which, when $s\ge \sqrt{3}$, as before gives
$$
\max_{\theta} (\ln \rho)''(\theta) < \frac{\sqrt{q}}{1+\sqrt{q}}~
$$
and when $1<s<\sqrt{3}$ gives
\begin{eqnarray*}
\max_{\theta} (\ln \rho)''(\theta) &<& \frac{2}{3 \sqrt{3}}\frac{A_q^2}{A_q^2-1} \\
&=&\frac{2}{3 \sqrt{3}}\frac{1}{1-(1-2(1+\sqrt{q})/3\sqrt{3q})} \\
&=&\frac{\sqrt{q}}{1+\sqrt{q}}~.
\end{eqnarray*}
In combination these two facts are sufficient to guarantee that \eqref{eq:StarBigCond1} is satisfied.	Theorem \ref{thm:disk} (for $q>1$) now follows from Theorem \ref{thm:main1Star}.
\end{proof}

\begin{proof}[Proof of Thoerem~\ref{thm:ellipse2focus}]
	Up to a rotational change of coordinates (placing it horizontally along its major axis and with the right-most focal point at the origin) we can express the ellipse $\Omega$ as a star-shaped domain with the radius function
	\begin{equation*}
		\rho(\theta)=\frac{a(1-e^2)}{1+e\cos\theta},
	\end{equation*}
	where $e$ is the eccentricity of the ellipse and $a$ is half of the major diameter.
	Then $$(\ln\rho)'(\theta)=\frac{e\sin\theta}{1+e\cos\theta}
	\AND (\ln\rho)''(\theta)=e\frac{e+\cos\theta}{\pare{1+e\cos\theta}^{2}}.$$
	 Aside from the admissibility condition \eqref{eq:StarBigCond1}, it is evident that the conditions
	in Theorem~\ref{thm:main1Star} are satisfied.
	It can easily be verified that
	\begin{equation*}
		\max_\theta\, (\ln\rho)''(\theta)=
		\begin{cases}
			(\ln\rho)''(0)=e/(1+e),	&\qquad\mbox{if $e<1/2$},
			\\
			(\ln\rho)''\Pare{\arccos(1/e-2e)}=1/(4-4e^2),	&\qquad\mbox{if $1/2\le e<1$}.
		\end{cases}
	\end{equation*}
	The admissible condition \eqref{eq:StarBigCond1} (since $q>1$) is now equivalent to
	\begin{equation*}
		4e^2<3-1/\sqrt{q}.
	\end{equation*}
With this observation Theorem \ref{thm:ellipse2focus} (for $q>1$) follows directly from Theorem \ref{thm:main1Star}.
\end{proof}

\section*{Acknowledgments}
The work of M.S. Vogelius was partially supported by NSF grant DMS-22-05912. The work of J. Xiao was partially supported by NSF grant DMS-23-07737.

\addcontentsline{toc}{section}{References}
\bibliographystyle{abbrv}
\bibliography{VXEllipse}

\end{document}